\documentclass[article,12pt]{amsart}

\usepackage{amsfonts}
\usepackage{amsmath}
\usepackage{graphicx}
\usepackage{color}
\usepackage{epsfig}
\usepackage{multirow}

\numberwithin{equation}{section}
\theoremstyle{plain}
\newtheorem{thm}{Theorem}[section]
\newtheorem{rem}{Remark}[section]

\newtheorem{cor}{Corollary}[section]
\newtheorem{lem}{Lemma}[section]

\newcommand{\dE}{\mathbb{E}}

\newcommand{\dZ}{\mathbb{Z}}
\newcommand{\dR}{\mathbb{R}}
\newcommand{\dC}{\mathbb{C}}

\newcommand{\dT}{\mathbb{T}}
\newcommand{\dV}{\mathbb{V}}
\newcommand{\cA}{\mathcal{A}}
\newcommand{\cL}{\mathcal{L}}
\newcommand{\cN}{\mathcal{N}}
\newcommand{\cF}{\mathcal{F}}
\newcommand{\cH}{\mathcal{H}}
\newcommand{\cP}{\mathcal{P}}
\newcommand{\cR}{\mathcal{R}}

\newcommand{\veps}{\varepsilon}
\newcommand{\wh}{\widehat}

\newcommand{\rI}{\mathrm{I}}
\newcommand{\ind}{\mbox{1}\kern-.25em \mbox{I}}

\newcommand{\tn}{\wh{\theta}_{n}}
\newcommand{\rn}{\wh{\rho}_{n}}
\newcommand{\dn}{\wh{D}_{n}}
\newcommand{\tk}{\wh{\theta}_{k}}
\newcommand{\rk}{\wh{\rho}_{k}}
\newcommand{\dk}{\wh{D}_{k}}
\newcommand{\e}{\wh{\veps}}
\newcommand{\en}{\wh{\veps}_{n}}
\newcommand{\ek}{\wh{\veps}_{k}}

\newcommand{\eek}{\wh{\veps}_{k-1}}
\newcommand{\VVert}{\vert \! \hspace{0.04cm} \vert \! \hspace{0.04cm} \vert}

\newcommand{\liml}{\overset{\cL}{\longrightarrow}}
\newcommand{\limp}{\overset{\cP}{\longrightarrow}}

\newcommand{\cvgps}{\hspace{0.3cm} \text{a.s.}}

\newcommand{\udots}{\mathinner{\mskip1mu\raise1pt\vbox{\kern7pt\hbox{.}}
  \mskip2mu\raise4pt\hbox{.}\mskip2mu\raise7pt\hbox{.}\mskip1mu}}

\email{Frederic.Proia@inria.fr}

\begin{document}

\title[Testing residuals from a stable autoregressive process]
{Further results on the H-Test of Durbin for stable autoregressive processes
\vspace{2ex}}
\address{Universit\'e Bordeaux 1, Institut de Math\'ematiques de Bordeaux,
UMR 5251, and INRIA Bordeaux, team ALEA, 351 Cours de la Lib\'eration, 33405 Talence cedex, France.}
\author{Fr\'ed\'eric Pro\"ia}
\thanks{}

\begin{abstract}
\textcolor{blue}{The purpose of this paper is to investigate the asymptotic behavior of the Durbin-Watson statistic for the stable $p-$order autoregressive process when the driven noise is given by a first-order autoregressive process. It is an extension of the previous work of Bercu and Pro\"ia devoted to the particular case $p=1$.} We establish the almost sure convergence and the asymptotic normality for both the least squares estimator of the unknown vector parameter of the autoregressive process as well as for the serial correlation estimator associated with the driven noise. In addition, the almost sure rates of convergence of our estimates are also provided. Then, we prove the almost sure convergence and the asymptotic normality for the Durbin-Watson statistic \textcolor{blue}{and we derive a two-sided statistical procedure for testing the presence of a significant first-order residual autocorrelation that appears to clarify and to improve the well-known \textit{h-test} suggested by Durbin. Finally, we briefly summarize our observations on simulated samples.}
\end{abstract}

\maketitle

\vspace{-0.5cm}

%%%%%%%%%%%%%%%%%%%%%%%%%%%%%%%%%%%%%%%%%%%%%%%%%%%%%%%%%%%%%%%%%%%%%%%%%%%%%%%%

\section{INTRODUCTION}

%%%%%%%%%%%%%%%%%%%%%%%%%%%%%%%%%%%%%%%%%%%%%%%%%%%%%%%%%%%%%%%%%%%%%%%%%%%%%%%%

The Durbin-Watson statistic was originally introduced by the eponymous econometricians Durbin and Watson \cite{DurbinWatson50}, \cite{DurbinWatson51}, \cite{DurbinWatson71} in the middle of last century, in order to detect the presence of a significant first-order autocorrelation in the residuals from a regression analysis. The statistical test worked pretty well in the independent framework of linear regression models, as it was specifically investigated by Tillman \cite{Tillman75}. While the Durbin-Watson statistic started to become well-known in Econometrics by being commonly used in the case of linear regression models containing lagged dependent random variables, Malinvaud \cite{Malinvaud61} and Nerlove and Wallis \cite{NerloveWallis66} observed that its widespread use in inappropriate situations were leading to inadequate conclusions. More precisely, they noticed that the Durbin-Watson statistic was asymptotically biased in the dependent framework. To remedy this misuse, alternative compromises were suggested. In particular, Durbin \cite{Durbin70} proposed a set of revisions of the original test, as the so-called \textit{t-test} and \textit{h-test}, and explained how to use them focusing on the first-order autoregressive process. It inspired a lot of works afterwards. More precisely, Maddala and Rao \cite{MaddalaRao73}, Park \cite{Park75} and then Inder \cite{Inder84}, \cite{Inder86} and Durbin \cite{Durbin86} looked into the approximation of the critical values and distributions under the null hypothesis, and showed by simulations that alternative tests significantly outperformed the inappropriate one, even on small-sized samples. Additional improvements were brought by King and Wu \cite{KingWu91} and lately, Stocker \cite{Stocker06} gave substantial contributions to the study of the asymptotic bias resulting from the presence of lagged dependent random variables. In most cases, the first-order autoregressive process was used as a reference for related research. This is the reason why the recent work of Bercu and Pro\"ia \cite{BercuProia11} was focused on such a process in order to give a new light on the distribution of the Durbin-Watson statistic under the null hypothesis as well as under the alternative hypothesis. They provided a sharp theoretical analysis rather than Monte-Carlo approximations, and they proposed a statistical procedure derived from the Durbin-Watson statistic. They showed how, from a theoretical and a practical point of view, this procedure outperforms the commonly used Box-Pierce \cite{BoxPierce70} and Ljung-Box \cite{LjungBox78} statistical tests, in the restrictive case of the first-order autoregressive process, even on small-sized samples. They also explained that such a procedure is asymptotically equivalent to the \textit{h-test} of Durbin \cite{Durbin70} for testing the significance of the first-order serial correlation. This work \cite{BercuProia11} had the ambition to bring the Durbin-Watson statistic back into light. It also inspired Bitseki Penda, Djellout and Pro\"ia \cite{BitsekiDjelloutProia12} who established moderate deviation principles on the least squares estimators and the Durbin-Watson statistic for the first-order autoregressive process where the driven noise is also given by a first-order autoregressive process.

\medskip

Our goal is to extend of the previous results of Bercu and Pro\"ia \cite{BercuProia11} to $p-$order autoregressive processes, contributing moreover to the investigation on several open questions left unanswered during four decades on the Durbin-Watson statistic \cite{Durbin70}, \cite{Durbin86}, \cite{NerloveWallis66}. One will observe that the multivariate framework is much more difficult to handle than the scalar case of \cite{BercuProia11}. We will focus our attention on the $p-$order autoregressive process given, for all $n\geq 1$, by 
\begin{equation}  
\label{Int_Mod}
\vspace{1ex}
\left\{
\begin{array}[c]{ccl}
X_{n} & = & \theta_1 X_{n-1} + \hdots + \theta_{\!p} X_{n-p} + \veps_n \vspace{1ex}\\
\veps_n & = & \rho \veps_{n-1} + V_n 
\end{array}
\right.
\end{equation}
where the unknown parameter $\theta = \begin{pmatrix} \theta_1 & \theta_2 & \hdots & \theta_{\!p} \end{pmatrix}^{\prime}$ is a nonzero vector such that $\Vert \theta \Vert_1 < 1$, and the unknown parameter $\vert \rho \vert < 1$.
Via an extensive use of the theory of martingales \cite{Duflo97}, \cite{HallHeyde80}, we shall provide a sharp and rigorous analysis on the asymptotic behavior of the least squares estimators of $\theta$ and $\rho$. The previous results of convergence were first established in probability \cite{Malinvaud61}, \cite{NerloveWallis66}, and more recently almost surely \cite{BercuProia11} in the particular case where $p=1$. We shall prove the almost sure convergence as well as the asymptotic normality of the least squares estimators of $\theta$ and $\rho$ in the more general multivariate framework, together with the almost sure rates of convergence of our estimates. We will deduce the almost sure convergence and the asymptotic normality for the Durbin-Watson statistic. Therefore, we shall be in the position to \textcolor{blue}{propose further results on the well-known \textit{h-test} of Durbin \cite{Durbin70} for testing the significance of the first-order serial correlation in the residuals.} We will also explain why, on the basis of the empirical power, this test procedure outperforms Ljung-Box \cite{LjungBox78} and Box-Pierce \cite{BoxPierce70} \textit{portmanteau} tests for stable autoregressive processes. We will finally show by simulation that it is equally powerful than the Breusch-Godfrey \cite{Breusch78}, \cite{Godfrey78} test and the \textit{h-test} \cite{Durbin70} on large samples, and better than all of them on small samples.

\medskip

The paper is organized as follows. Section 2 is devoted to the estimation of the autoregressive parameter. We establish the almost sure convergence of the least squares vector estimator of $\theta$ to the limiting value 
\begin{equation}
\label{Int_Tlim}
\theta^{*} = \alpha \left( I_{\! p} - \theta_{\!p} \rho J_{\! p} \right) \beta
\end{equation}
where $I_{\! p}$ is the identity matrix of order $p$, $J_{\! p}$ is the exchange matrix of order $p$, and where $\alpha$ and $\beta$ will be calculated explicitly. The asymptotic normality as well as the quadratic strong law and a set of results derived from the law of iterated logarithm are provided. Section 3 deals with the estimation of the serial correlation parameter. The almost sure convergence of the least squares estimator of $\rho$ to
\begin{equation}
\label{Int_Rlim}
\rho^* = \theta_{\!p} \rho \theta_{\!p}^{*}
\end{equation}
where $\theta_{\!p}^{*}$ stands for the $p-$th component of $\theta^{*}$ is also established along with the quadratic strong law, the law of iterated logarithm and the asymptotic normality. It enables us to establish in Section 4 the almost sure convergence of the Durbin-Watson statistic to 
\begin{equation}
\label{Int_Dlim}
D^* =  2(1 - \rho^*)
\end{equation}
together with its asymptotic normality. Our sharp analysis on the asymptotic behavior of the Durbin-Watson statistic remains true whatever the values of the parameters $\theta$ and $\rho$ as soon as $\Vert \theta \Vert_1 < 1$ and $\vert \rho \vert < 1$, assumptions resulting from the stability of the model. Consequently, we are able in Section 4 to propose a \textcolor{blue}{two-sided statistical test for the presence of a significant first-order residual autocorrelation closely related to the \textit{h-test} of Durbin \cite{Durbin70}. A theoretical comparison as well as a sharp analysis of both approaches are also provided. In Section 5, we give a short conclusion where we briefly summarize our observations on simulated samples. We compare the empirical power of this test procedure with the commonly used \textit{portmanteau} tests of Box-Pierce \cite{BoxPierce70} and Ljung-Box \cite{LjungBox78}, with the Breusch-Godfrey test \cite{Breusch78}, \cite{Godfrey78} and the \textit{h-test} of Durbin \cite{Durbin70}.} Finally, the proofs related to linear algebra calculations are postponed in Appendix A and all the technical proofs of Sections 2 and 3 are postponed in Appendices B and C, respectively. \textcolor{blue}{Moreover, Appendix D is devoted to the asymptotic equivalence between the \textit{h-test} of Durbin and our statistical test procedure.}

\begin{rem}
\label{Int_Rem_DefIJe}
In the whole paper, for any matrix $M$, $M^{\prime}$ is the transpose of $M$. For any square matrix $M$, $\textnormal{tr}(M)$, $\det(M)$, $\VVert M \VVert_1$ and $\rho(M)$ are the trace, the determinant, the 1-norm and the spectral radius of $M$, respectively. In addition, $\lambda_{\text{min}}(M)$ and $\lambda_{\text{max}}(M)$ denote the smallest and the largest eigenvalues of $M$, respectively. For any vector $v$, $\Vert v \Vert$ stands for the euclidean norm of $v$ and $\Vert v \Vert_1$ is the 1-norm of $v$.
\end{rem}

\begin{rem}
Before starting, we denote by $I_{\! p}$ be the identity matrix of order $p$, $J_{\! p}$ the exchange matrix of order $p$ and $e$ the $p-$dimensional vector given by
\begin{equation*}
I_{\! p} = \begin{pmatrix}
1 & 0 & \hdots & 0\\
0 & 1 & \hdots & 0\\
\vdots & \vdots & \ddots & \vdots\\
0 & 0 & \hdots & 1
\end{pmatrix}, \hspace{1cm}
J_{\! p} = \begin{pmatrix}
0 & \hdots & 0 & 1\\
0 & \hdots & 1 & 0\\
\vdots & \udots & \vdots & \vdots\\
1 & \hdots & 0 & 0
\end{pmatrix}, \hspace{1cm}
e = \begin{pmatrix}
1\\
0\\
\vdots\\
0
\end{pmatrix}.
\end{equation*}
\end{rem}

\bigskip

%%%%%%%%%%%%%%%%%%%%%%%%%%%%%%%%%%%%%%%%%%%%%%%%%%%%%%%%%%%%%%%%%%%%%%%%%%%%%%%%

\section{ON THE AUTOREGRESSIVE PARAMETER}

%%%%%%%%%%%%%%%%%%%%%%%%%%%%%%%%%%%%%%%%%%%%%%%%%%%%%%%%%%%%%%%%%%%%%%%%%%%%%%%%

Consider the $p-$order autoregressive process given by \eqref{Int_Mod} where we shall suppose, to make calculations lighter without loss of generality, that the square-integrable initial values $X_0 = \veps_0$ and $X_{-1}, X_{-2}, \hdots, X_{-p} = 0$. In all the sequel, we assume that $(V_{n})$ is a sequence of square-integrable, independent and identically distributed random variables with zero mean and variance $\sigma^2 > 0$. Let us start by introducing some notations. Let $\Phi^p_{n}$ stand for the lag vector of order $p$, given for all $n \geq 0$, by
\begin{equation}  
\label{P1_Phi_Lag}
\Phi^p_{n} = \begin{pmatrix} X_{n} & \hspace{0.1cm} & X_{n-1} & \hspace{0.1cm} & \hdots & \hspace{0.1cm} & X_{n-p+1} \end{pmatrix}^{\prime}.
\end{equation}
Denote by $S_{n}$ the positive definite matrix defined, for all $n \geq 0$, as
\begin{equation}
\label{P1_Sn}
S_{n} = \sum_{k=0}^{n} \Phi^p_{k}\, {\Phi^p_{k}}^{\, \prime} + S
\end{equation}
where the symmetric and positive definite matrix $S$ is added in order to avoid an useless invertibility assumption. For the estimation of the unknown parameter $\theta$, it is natural to make use of the least squares estimator which minimizes
\begin{equation*}
\nabla_{\!n}(\theta) = \sum_{k=1}^n \left( X_k - \theta^{\, \prime}\, \Phi^p_{k-1} \right)^2.
\end{equation*}
A standard calculation leads, for all $n \geq 1$, to
\begin{equation}  
\label{P1_Est}
\tn = (S_{n-1})^{\! -1} \sum_{k=1}^{n} \Phi^p_{k-1}\, X_{k}.
\end{equation}

\medskip

\noindent Our first result is related to the almost sure convergence of $\tn$ to the limiting value $\theta^{*} =  \alpha \left( I_{\! p} - \theta_{\!p} \rho J_{\! p} \right) \beta$, where
\begin{equation}
\label{P1_Alpha}
\alpha = \frac{1}{(1 - \theta_{\!p} \rho)(1 + \theta_{\!p} \rho)},
\end{equation}
\begin{equation}
\label{P1_Beta}
\beta =
\begin{pmatrix}
\theta_1 + \rho & \hspace{0.1cm} & \theta_2 - \theta_1 \rho & \hspace{0.1cm} & \hdots & \hspace{0.1cm} & \theta_{\!p} - \theta_{\!p-1}\rho \end{pmatrix}^{\prime}.
\end{equation}

\begin{thm}
\label{P1_Thm_CvgTheta}
We have the almost sure convergence
\begin{equation}  
\label{P1_CvgTheta}
\lim_{n \rightarrow \infty} \tn = \theta^{*} \textnormal{\cvgps}
\end{equation}
\end{thm}

\begin{rem}
In the particular case where $\rho=0$, we obtain the strong consistency of the least squares estimate in a stable autoregressive model, already proved \textnormal{e.g.} in \cite{LaiWei83}, under the condition of stability $\Vert \theta \Vert_1 < 1$.
\end{rem}

\medskip

\noindent Let us now introduce the square matrix B of order $p+2$, partially made of the elements of $\beta$ given by \eqref{P1_Beta},
\begin{equation}
\label{P1_B}
B =
\begin{pmatrix}
1 & -\beta_1 & -\beta_2 & \hdots & \hdots & \hspace{0.15cm} -\beta_{p-1} \hspace{0.15cm} & \hspace{0.15cm} -\beta_p \hspace{0.15cm} & \hspace{0.15cm} \theta_{\!p} \rho \hspace{0.15cm}\\
-\beta_1 & 1-\beta_2 & -\beta_3 & \hdots & \hdots & -\beta_p & \theta_{\!p} \rho & 0\\
-\beta_2 & -\beta_1-\beta_3 & \hspace{0.15cm} 1-\beta_4 \hspace{0.15cm} & \hdots & \hdots & \theta_{\!p} \rho & 0 & 0\\
\vdots & \vdots & \vdots & & & \vdots & \vdots & \vdots\\
\vdots & \vdots & \vdots & & & \vdots & \vdots & \vdots\\
\hspace{0.15cm} -\beta_p \hspace{0.15cm} & \hspace{0.15cm} -\beta_{p-1}+\theta_{\!p} \rho \hspace{0.15cm} & -\beta_{p-2} & \hdots & \hdots & -\beta_1 & 1 & 0\\
\theta_{\!p} \rho & -\beta_p & -\beta_{p-1} & \hdots & \hdots & -\beta_2 & -\beta_1 & 1\\
\end{pmatrix}.
\end{equation}

\medskip

\noindent Under our stability conditions, we are able to establish the invertibility of $B$ in Lemma \ref{P1_Lem_InvB}. The corollary that follows will be useful in the next section.
\begin{lem}
\label{P1_Lem_InvB}
Under the stability conditions $\Vert \theta \Vert_1 < 1$ and $\vert \rho \vert < 1$, the matrix $B$ given by \eqref{P1_B} is invertible. 
\end{lem}
\begin{cor}
\label{P1_Cor_InvC}
By virtue of Lemma \ref{P1_Lem_InvB}, the submatrix $C$ obtained by removing from $B$ its first row and first column is invertible.
\end{cor}
\noindent From now on, $\Lambda \in \dR^{p+2}$ is the unique solution of the linear system $B \Lambda = e$, \textit{i.e.}
\begin{equation}
\label{P1_VecLim}
\Lambda = B^{-1} e
\end{equation}
where the vector $e$ has already been defined in Remark \ref{Int_Rem_DefIJe}, but in higher dimension. Denote by $\lambda_0, \hdots, \lambda_{p+1}$ the elements of $\Lambda$ and let $\Delta_{p}$ be the Toeplitz matrix of order $p$ associated with the first $p$ elements of $\Lambda$, that is
\begin{equation}
\label{P1_Lambda}
\Delta_{p} = \begin{pmatrix}
\lambda_0 & \lambda_1 & \lambda_2 & \hdots & \hdots & \lambda_{p-1}\\
\lambda_1 & \lambda_0 & \lambda_1 & \hdots & \hdots & \lambda_{p-2}\\
\vdots & \vdots & \vdots & & & \vdots\\
 \vdots & \vdots & \vdots & & & \vdots\\
\lambda_{p-1} & \lambda_{p-2} & \lambda_{p-3} & \hdots & \hdots & \lambda_0\\
\end{pmatrix}.
\end{equation}

\medskip

\noindent Via the same lines, we are able to establish the invertibility of $\Delta_{p}$ in Lemma \ref{P1_Lem_InvL}.
\begin{lem}
\label{P1_Lem_InvL}
Under the stability conditions $\Vert \theta \Vert_1 < 1$ and $\vert \rho \vert < 1$, for all $p \geq 1$, the matrix $\Delta_{p}$ given by \eqref{P1_Lambda} is positive definite.
\end{lem}
\noindent In light of foregoing, our next result deals with the asymptotic normality of $\tn$.
\begin{thm}
\label{P1_Thm_TlcTheta}
Assume that $(V_{n})$ has a finite moment of order 4. Then, we have
the asymptotic normality
\begin{equation}  
\label{P1_TlcTheta}
\sqrt{n} \left( \tn - \theta^{*} \right) \liml \cN ( 0, \Sigma_\theta)
\end{equation}
where the asymptotic covariance matrix is given by
\begin{equation}
\label{P1_SigT}
\Sigma_{\theta} = \alpha^2 \left( I_{\! p} - \theta_{\!p} \rho J_{\! p} \right) \Delta_{p}^{-1} \left( I_{\! p} - \theta_{\!p} \rho J_{\! p} \right).
\end{equation}
\end{thm}

\begin{rem}
The covariance matrix $\Sigma_{\theta}$ is invertible under the stability conditions. Furthermore, due to the way it is constructed, $\Sigma_{\theta}$ is bisymmetric.
\end{rem}

\begin{rem}
\textcolor{blue}{In the particular case where $\rho=0$, $\Sigma_{\theta}$ reduces to $\Delta_{p}^{-1}$. This is a well-known result related to the asymptotic normality of the Yule-Walker estimator for the causal autoregressive process that can be found \textnormal{e.g.} in Theorem 8.1.1 of \cite{BrockwellDavis91}.}
\end{rem}

\noindent After establishing the almost sure convergence of the estimator $\tn$ and its asymptotic normality, we focus our attention on the almost sure rates of convergence.

\begin{thm}
\label{P1_Thm_RatTheta}
Assume that $(V_{n})$ has a finite moment of order 4. Then, we have
the quadratic strong law
\begin{equation}
\label{P1_LfqTheta}
\lim_{n \rightarrow \infty} \frac{1}{\log n} \sum_{k=1}^{n} \left( \tk - \theta^{*} \right) \left( \tk - \theta^{*} \right)^{\prime} = \Sigma_{\theta} \textnormal{\cvgps}
\end{equation}
where $\Sigma_{\theta}$ is given by \eqref{P1_SigT}. In addition, for all $v \in \dR^{p}$, we also have the law of iterated logarithm
\begin{eqnarray}
\label{P1_LliTheta}
\limsup_{n \rightarrow \infty} \left( \frac{n}{2 \log \log n} \right)^{\! 1/2} v^{\, \prime} \left( \tn - \theta^{*} \right) & = & -\liminf_{n \rightarrow \infty} \left( \frac{n}{2 \log \log n} \right)^{\! 1/2} v^{\, \prime} \left( \tn - \theta^{*} \right), \nonumber\\
 & = & \sqrt{v^{\, \prime}\, \Sigma_{\theta}\, v} \textnormal{\cvgps}
\end{eqnarray}
Consequently,
\begin{equation}
\label{P1_LliTheta2}
\limsup_{n \rightarrow \infty} \left( \frac{n}{2 \log \log n} \right) \left( \tn - \theta^{*} \right) \left( \tn - \theta^{*} \right)^{\prime} = \Sigma_{\theta} \textnormal{\cvgps}
\end{equation}
In particular,
\begin{equation}
\label{P1_LliTheta3}
\limsup_{n \rightarrow \infty} \left( \frac{n}{2 \log \log n} \right) \big\Vert \tn - \theta^{*} \big\Vert^2 = \textnormal{tr}(\Sigma_{\theta}) \textnormal{\cvgps}
\end{equation}
\end{thm}

\begin{rem}
\label{P1_Rem_RatTheta}
It clearly follows from \eqref{P1_LfqTheta} that
\begin{equation}
\label{P1_LfqTraceTheta}
\lim_{n \rightarrow \infty} \frac{1}{\log n} \sum_{k=1}^{n} \big\Vert \tk - \theta^{*} \big\Vert^2 = \textnormal{tr}(\Sigma_{\theta}) \textnormal{\cvgps}
\end{equation}
Furthermore, from \eqref{P1_LliTheta3}, we have the almost sure rate of convergence
\begin{equation}
\label{P1_RatCvgTheta}
\big\Vert \tn - \theta^{*} \big\Vert^2 = O\left( \frac{\log \log n}{n} \right) \textnormal{\cvgps}
\end{equation}
\end{rem}

\begin{proof}
The proofs of Lemma \ref{P1_Lem_InvB} and Lemma \ref{P1_Lem_InvL} are given in Appendix A while those of Theorems \ref{P1_Thm_CvgTheta} to \ref{P1_Thm_RatTheta} may be found in Appendix B.
\end{proof}

\medskip

\noindent To conclude this section, let us draw a parallel between the results of \cite{BercuProia11} and the latter results for $p=1$. In this particular case, $\beta$ and $\alpha$ reduce to $(\theta+\rho)$ and $(1-\theta \rho)^{-1}(1+\theta\rho)^{-1}$ respectively, and it is not hard to see that we obtain the almost sure convergence of our estimate to
\begin{equation*}
\theta^{*} = \frac{\theta + \rho}{1 + \theta \rho}.
\end{equation*}
In addition, a straightforward calculation leads to
\begin{equation*}
\Sigma_{\theta} = \frac{(1 - \theta^2)(1 - \theta \rho)(1 - \rho^2)}{(1 + \theta \rho)^3}.
\end{equation*}
One can verify that these results correspond to Theorem 2.1 and Theorem 2.2 of \cite{BercuProia11}.

\bigskip

%%%%%%%%%%%%%%%%%%%%%%%%%%%%%%%%%%%%%%%%%%%%%%%%%%%%%%%%%%%%%%%%%%%%%%%%%%%%%%%%

\section{ON THE SERIAL CORRELATION PARAMETER}

%%%%%%%%%%%%%%%%%%%%%%%%%%%%%%%%%%%%%%%%%%%%%%%%%%%%%%%%%%%%%%%%%%%%%%%%%%%%%%%%

This section is devoted to the estimation of the serial correlation parameter $\rho$. First of all, it is necessary to evaluate, at stage $n$, the residual set $(\en)$ resulting from the biased estimation of $\theta$. For all $1 \leq k \leq n$, let
\begin{equation}
\label{P2_EstRes}
\ek = X_{k} - \tn^{\: \prime}\, \Phi_{k-1}^p.
\end{equation}
The initial value $\e_0$ may be arbitrarily chosen and we take $\e_0 = X_0$ for a matter of simplification. Then, a natural way to estimate $\rho$ is to make use of the least squares estimator which minimizes
\begin{equation*}
\nabla_{\!n}(\rho) = \sum_{k=1}^n \big( \ek - \rho\, \eek \big)^2.
\end{equation*}
Hence, it clearly follows that, for all $n \geq 1$,
\begin{equation}
\label{P2_Est}
\rn = \left( \sum_{k=1}^{n} \eek^{~ 2} \right)^{\! -1} \sum_{k=1}^{n} \ek\, \eek.
\end{equation}
It is important to note that one deals here with a scalar problem, in contrast to the study of the estimator of $\theta$ in Section 2. Our goal is to obtain the same asymptotic properties for the estimator of $\rho$ as those obtained for each component of the one of $\theta$. However, one shall realize that the results of this section are much more tricky to establish than those of the previous one.

\medskip

\noindent We first state the almost sure convergence of $\rn$ to the limiting value $\rho^{*} = \theta_{\!p} \rho \theta_{\!p}^{*}$.

\begin{thm}
\label{P2_Thm_CvgRho}
We have the almost sure convergence
\begin{equation}  
\label{P2_CvgRho}
\lim_{n \rightarrow \infty} \rn = \rho^{*} \textnormal{\cvgps}
\end{equation}
\end{thm}

\noindent Our next result deals with the joint asymptotic normality of $\tn$ and $\rn$. For that purpose, it is necessary to introduce some additional notations. Denote by $P$ the square matrix of order $p+1$ given by
\begin{equation}
\label{P2_P}
P = \begin{pmatrix}
P_{B} & 0\\
P_{L}^{\, \prime} & \varphi
\end{pmatrix}
\end{equation}
where
\begin{eqnarray*}
P_{B} & = & \alpha \big( I_{\! p} - \theta_{\!p} \rho J_{\! p} \big) \Delta_{p}^{-1},\\
P_{L} & = & J_{\! p} \big( I_{\! p} - \theta_{\!p} \rho J_{\! p} \big) \big( \alpha \theta_{\!p} \rho\,  \Delta_{p}^{-1} e + \theta_{\!p}^{*}\, \beta \big),\\
\varphi & = & - \alpha^{-1} \theta_{\!p}^{*}.
\end{eqnarray*}
Furthermore, let us introduce the Toeplitz matrix $\Delta_{p+1}$ of order $p+1$ which is the extension of $\Delta_{p}$ given by \eqref{P1_Lambda} to the next dimension,
\begin{equation}
\label{P2_LambdaP}
\Delta_{p+1} = \begin{pmatrix}
\Delta_{p} & J_{\! p}\, \Lambda_{p}^{1}\\
{\Lambda_{p}^{1}}^{\, \prime} J_{\! p} & \lambda_0
\end{pmatrix}
\end{equation}
with $\Lambda_{p}^{1} = \begin{pmatrix} \lambda_1 & \lambda_2 & \hdots & \lambda_{p} \end{pmatrix}^{\prime}$, and the positive semidefinite covariance matrix $\Gamma$ of order $p+1$, given by
\begin{equation}
\label{P2_GamCov}
\Gamma = P \Delta_{p+1} P^{\, \prime}.
\end{equation}

\begin{thm}
\label{P2_Thm_TlcRho}
Assume that $(V_{n})$ has a finite moment of order 4. Then, we have
the joint asymptotic normality
\begin{equation}
\label{P2_TlcJoint}
\sqrt{n} \begin{pmatrix}
\tn - \theta^{*}\\
\rn - \rho^{*}
\end{pmatrix} \liml \cN ( 0, \Gamma).
\end{equation}
In particular,
\begin{equation}
\label{P2_TlcRho}
\sqrt{n} \Big( \rn - \rho^{*} \Big) \liml \cN ( 0, \sigma^2_{\rho})
\end{equation}
where $\sigma^2_{\rho} = \Gamma_{p+1,\: p+1}$ is the last diagonal element of $\Gamma$.
\end{thm}

\begin{rem}
The covariance matrix $\Gamma$ has the following explicit expression,
\begin{equation*}
\Gamma = \begin{pmatrix}
\Sigma_{\theta} & \theta_{\!p} \rho\, J_{\! p}\, \Sigma_{\theta}\, e\\
\theta_{\!p} \rho\, e^{\, \prime} \Sigma_{\theta} J_{\! p}  & \sigma^2_{\rho}
\end{pmatrix}
\end{equation*}
where
\begin{equation}
\label{P2_SigR}
\sigma_{\rho}^2 = P_{L}^{\, \prime}\, \Delta_{p}\, P_{L} - 2 \alpha^{-1} \theta_{\!p}^{*} {\Lambda_{p}^{1}}^{\, \prime} J_{\! p}\, P_{L} + \left( \alpha^{-1} \theta_{\!p}^{*} \right)^2 \lambda_0.
\end{equation}
\end{rem}

\begin{rem}
\label{P2_Rem_InvGamma}
The covariance matrix $\Gamma$ is invertible under the stability conditions if and only if $\theta_{\!p}^{*} \neq 0$ since, by a straightforward calculation,
\begin{equation*}
\det(\Gamma) = \alpha^{2(p-1)} \left( \theta_{\! p}^{*} \right)^2 \det(\Delta_{p+1}) \left( \frac{\det(I_{\! p} - \theta_{\!p} \rho J_{\! p})}{\det(\Delta_{p})} \right)^2
\end{equation*}
according to Lemma \ref{P1_Lem_InvL} and noticing that $(I_{\! p} - \theta_{\!p} \rho J_{\! p})$ is strictly diagonally dominant, thus invertible. As a result, the joint asymptotic normality given by \eqref{P2_TlcJoint} is degenerate in any situation such that $\theta_{\! p}^{*}=0$, that is
\begin{equation}
\label{JointCLT_Assumption}
\theta_{\!p} - \theta_{\!p-1} \rho = \theta_{\!p} \rho (\theta_{1} + \rho).
\end{equation}
Moreover, \eqref{P2_TlcRho} holds on $\{ \theta_{\!p} - \theta_{\!p-1} \rho \neq \theta_{\!p} \rho (\theta_{1} + \rho)\} \cup \{ \theta_{\!p} \neq 0, \rho \neq 0 \}$, otherwise the asymptotic normality associated with $\rn$ is degenerate. In fact, a more restrictive condition ensuring that \eqref{P2_TlcRho} still holds may be $\{ \theta_{\!p} \neq 0 \}$, \textit{i.e.} that one deals at least with a $p-$order autoregressive process. This restriction seems natural in the context of the study and can be compared to the assumption $\{ \theta \neq 0 \}$ in \cite{BercuProia11}. Theorem 3.2 of \cite{BercuProia11} ensures that the joint asymptotic normality is degenerate under $\{ \theta = -\rho \}$. One can note that such an assumption is equivalent to \eqref{JointCLT_Assumption} in the case of the $p-$order process, since both of them mean that the last component of $\theta^{*}$ has to be nonzero.
\end{rem}

\noindent The almost sure rates of convergence for $\rn$ are as follows.

\begin{thm}
\label{P2_Thm_RatRho}
Assume that $(V_{n})$ has a finite moment of order 4. Then, we have
the quadratic strong law
\begin{equation}
\label{P2_LfqRho}
\lim_{n \rightarrow \infty} \frac{1}{\log n} \sum_{k=1}^{n} \Big( \rk - \rho^{*} \Big)^2 = \sigma^2_{\rho} \textnormal{\cvgps}
\end{equation}
where $\sigma^2_{\rho}$ is given by \eqref{P2_SigR}. In addition, we also have the law of iterated logarithm
\begin{eqnarray}
\label{P2_LliRho}
\limsup_{n \rightarrow \infty} \left( \frac{n}{2 \log \log n} \right)^{\! 1/2} \Big( \rn - \rho^{*} \Big) & = & -\liminf_{n \rightarrow \infty} \left( \frac{n}{2 \log \log n} \right)^{\! 1/2} \Big( \rn - \rho^{*} \Big), \nonumber\\
 & = & \sigma_{\rho} \textnormal{\cvgps}
\end{eqnarray}
Consequently,
\begin{equation}
\label{P2_LliRho2}
\limsup_{n \rightarrow \infty} \left( \frac{n}{2 \log \log n} \right) \Big( \rn - \rho^{*} \Big)^2 = \sigma_{\rho}^2 \textnormal{\cvgps}
\end{equation}
\end{thm}

\begin{rem}
\label{P2_Rem_RatRho}
It clearly follows from \eqref{P2_LliRho2} that we have the almost sure rate of convergence
\begin{equation}
\label{P2_RatCvgRho}
\Big( \rn - \rho^{*} \Big)^2 = O\left( \frac{\log \log n}{n} \right) \textnormal{\cvgps}
\end{equation}
\end{rem}

\medskip

\noindent As before, let us also draw the parallel between the results of \cite{BercuProia11} and the latter results for $p=1$. In this particular case, we immediately obtain $\rho^{*} = \theta \rho \theta^{*}$. Moreover, an additionnal step of calculation shows that
\begin{equation*}
\sigma_{\rho}^2 = \frac{1 - \theta\rho}{(1 + \theta\rho)^3} \left( (\theta + \rho)^2 (1 + \theta\rho)^2 + (\theta \rho)^2 (1 - \theta^2) (1 - \rho^2) \right).
\end{equation*}
One can verify that these results correspond to Theorem 3.1 and Theorem 3.2 of \cite{BercuProia11}. Besides, the estimators of $\theta$ and $\rho$ are self-normalized. Consequently, the asymptotic variances $\Sigma_{\theta}$ and $\sigma_{\rho}^2$ do not depend on the variance $\sigma^2$ associated with the driven noise $(V_{n})$. To be complete and provide an important statistical aspect, it seemed advisable to suggest an estimator of the true variance $\sigma^2$ of the model, based on these previous estimates. Consider, for all $n \geq 1$, the estimator given by
\begin{equation}
\label{P2_EstSig}
\wh{\sigma}_{n}^{\, 2} = \left( 1 - \rn^{~ 2}\, \wh{\theta}_{\!p,\, n}^{\, -2} \right) \frac{1}{n} \sum_{k=0}^{n} \ek^{~ 2}
\end{equation}
where $\wh{\theta}_{\!p,\, n}$ stands for the $p-$th component of $\tn$.

\begin{thm}
\label{P2_Thm_CvgSig}
We have the almost sure convergence
\begin{equation}  
\label{P2_CvgRho}
\lim_{n \rightarrow \infty} \wh{\sigma}_{n}^{\, 2} = \sigma^2 \textnormal{\cvgps}
\end{equation}
\end{thm}

\begin{proof}
The proofs of Theorems \ref{P2_Thm_CvgRho} to \ref{P2_Thm_RatRho} are given in Appendix C. The one of Theorem \ref{P2_Thm_CvgSig} is left to the reader as it directly follows from that of Theorem \ref{P2_Thm_CvgRho}.
\end{proof}

\bigskip

%%%%%%%%%%%%%%%%%%%%%%%%%%%%%%%%%%%%%%%%%%%%%%%%%%%%%%%%%%%%%%%%%%%%%%%%%%%%%%%%

\section{ON THE DURBIN-WATSON STATISTIC}

%%%%%%%%%%%%%%%%%%%%%%%%%%%%%%%%%%%%%%%%%%%%%%%%%%%%%%%%%%%%%%%%%%%%%%%%%%%%%%%%

We shall now investigate the asymptotic behavior of the Durbin-Watson statistic for the general autoregressive process \cite{DurbinWatson50}, \cite{DurbinWatson51}, \cite{DurbinWatson71}, given, for all $n \geq 1$, by
\begin{equation}
\label{P3_Dn}
\dn = \left( \sum_{k=0}^{n} \ek^{~ 2} \right)^{\! -1} \sum_{k=1}^{n} \Big( \ek - \eek \Big)^2.
\end{equation}
As mentioned, the almost sure convergence and the asymptotic normality of the Durbin-Watson statistic have previously been investigated in \cite{BercuProia11} in the particular case where $p=1$. It has enabled the authors to propose a two-sided statistical test for the presence of a significant residual autocorrelation. They also explained how this statistical procedure outperformed the commonly used Ljung-Box \cite{LjungBox78} and Box-Pierce \cite{BoxPierce70} \textit{portmanteau} tests for white noise in the case of the first-order autoregressive process, and how it was asymptotically equivalent to the \textit{h-test} of Durbin \cite{Durbin70}, on a theoretical basis and on simulated data. They went even deeper in the study, establishing the distribution of the statistic under the null hypothesis $`` \rho = \rho_0"$, with $\vert \rho_0 \vert < 1$, as well as under the alternative hypothesis $`` \rho \neq \rho_0"$, and noticing the existence of a critical situation in the case where $\theta = -\rho$. This pathological case arises when the covariance matrix $\Gamma$ given by \eqref{P2_GamCov} is singular, and can be compared in the multivariate framework to the content of Remark \ref{P2_Rem_InvGamma}. Our goal is to obtain the same asymptotic results for all $p \geq 1$ so as to build a new statistical procedure for testing serial correlation in the residuals. In this paper, we shall only focus our attention on the test $`` \rho = 0"$ against $`` \rho \neq 0"$, of increased statistical interest. \textcolor{blue}{We shall see below that from a theoretical and a practical point of view, our statistical test procedure clarifies ans outperforms the \textit{h-test} of Durbin. In particular, it avoids the presence of an abstract variance estimation likely to generate perturbations on small-sized samples.} In the next section, \textcolor{blue}{we will observe on simulated data that the procedure proposed in Theorem \ref{P3_Thm_TestD} is more powerful than the \textit{portmanteau} tests \cite{LjungBox78}, \cite{BoxPierce70}, often used} for testing the significance of the first-order serial correlation of the driven noise in a $p-$order autoregressive process.

\medskip

\noindent First, one can observe that $\dn$ and $\rn$ are asymptotically linked together by an affine transformation. Consequently, the asymptotic behavior of the Durbin-Watson statistic directly follows from the previous section. We start with the almost sure convergence to the limiting value $D^{*} = 2(1 - \rho^{*})$.

\begin{thm}
\label{P3_Thm_CvgD}
We have the almost sure convergence
\begin{equation}  
\label{P3_CvgD}
\lim_{n \rightarrow \infty} \dn = D^{*} \textnormal{\cvgps}
\end{equation}
\end{thm}

\noindent Our next result deals with the asymptotic normality of $\dn$. It will be the keystone of the statistical procedure deciding whether residuals have a significant first-order correlation or not, for a given significance level. Denote
\begin{equation}
\label{P3_SigD}
\sigma_{D}^2 = 4 \sigma_{\rho}^2
\end{equation}
where the variance $\sigma_{\rho}^2$ is given by \eqref{P2_SigR}.

\begin{thm}
\label{P3_Thm_TlcD}
Assume that $(V_{n})$ has a finite moment of order 4. Then, we have the asymptotic normality
\begin{equation}
\label{P3_TlcD}
\sqrt{n} \left( \dn - D^{*} \right) \liml \cN ( 0, \sigma^2_{D}).
\end{equation}
\end{thm}

\begin{rem}
We immediately deduce from \eqref{P3_TlcD} that
\begin{equation}
\label{P2_TlcDChi2}
\frac{n}{\sigma^2_{D}} \left( \dn - D^{*} \right)^2 \liml \chi^2
\end{equation}
where $\chi^2$ has a Chi-square distribution with one degree of freedom.
\end{rem}

\noindent Let us focus now on the almost sure rates of convergence of $\dn$.

\begin{thm}
\label{P3_Thm_RatD}
Assume that $(V_{n})$ has a finite moment of order 4. Then, we have
the quadratic strong law
\begin{equation}
\label{P3_LfqD}
\lim_{n \rightarrow \infty} \frac{1}{\log n} \sum_{k=1}^{n} \left( \dk - D^{*} \right)^2 = \sigma^2_{D} \textnormal{\cvgps}
\end{equation}
where $\sigma_{D}^2$ is given by \eqref{P3_SigD}. In addition, we also have the law of iterated logarithm
\begin{eqnarray}
\label{P3_LliD}
\limsup_{n \rightarrow \infty} \left( \frac{n}{2 \log \log n} \right)^{\! 1/2} \left( \dn - D^{*} \right) & = & -\liminf_{n \rightarrow \infty} \left( \frac{n}{2 \log \log n} \right)^{\! 1/2} \left( \dn - D^{*} \right), \nonumber\\
 & = & \sigma_{D} \textnormal{\cvgps}
\end{eqnarray}
Consequently,
\begin{equation}
\label{P3_LliD2}
\limsup_{n \rightarrow \infty} \left( \frac{n}{2 \log \log n} \right) \left( \dn - D^{*} \right)^2 = \sigma_{D}^2 \textnormal{\cvgps}
\end{equation}
\end{thm}

\begin{rem}
\label{P3_Rem_RatD}
It clearly follows from \eqref{P3_LliD2} that we have the almost sure rate of convergence
\begin{equation}
\label{P3_RatCvgD}
\left( \dn - D^{*} \right)^2 = O\left( \frac{\log \log n}{n} \right) \textnormal{\cvgps}
\end{equation}
\end{rem}

\noindent We are now in the position to propose the two-sided statistical test built on the Durbin-Watson statistic. First of all, we shall not investigate the particular case where $\theta_{\!p} = 0$ since our procedure is of interest only for autoregressive processes of order $p$. One wishes to test the presence of a significant serial correlation, setting
$$\cH_0\,:\,`` \rho = 0" \hspace{1cm} \text{against}\hspace{1cm} \cH_1\,:\,`` \rho \neq 0".
\vspace{2ex}$$

\begin{thm}
\label{P3_Thm_TestD}
Assume that $(V_{n})$ has a finite moment of order 4, $\theta_{\!p} \neq 0$ and $\theta_{\!p}^{*} \neq 0$. Then, under the null hypothesis $\cH_0\,:\,`` \rho = 0"$,
\begin{equation}
\label{P3_TestH0}
\frac{n}{4 \wh{\theta}_{\!p,\, n}^{\, 2}} \left( \dn - 2 \right)^2 \liml \chi^2
\end{equation}
where $\wh{\theta}_{\!p,\, n}$ stands for the $p-$th component of $\tn$, and where $\chi^2$ has a Chi-square distribution with one degree of freedom. In addition, under the alternative hypothesis $\cH_1\,:\,`` \rho \neq 0"$,
\begin{equation}
\label{P3_TestH1}
\lim_{n \rightarrow \infty} \frac{n}{4 \wh{\theta}_{\!p,\, n}^{\, 2}} \left( \dn - 2 \right)^2 = +\infty \textnormal{\cvgps}
\end{equation}
\end{thm}
\noindent From a practical point of view, for a significance level $a$ where $0 < a < 1$, the acceptance and rejection regions are given by $\cA = [0, z_{a}]$ and $\cR = \hspace{0.15cm} ]z_{a}, +\infty[$ where $z_{a}$ stands for the $(1-a)-$quantile of the Chi-square distribution with one degree of freedom. The null hypothesis $\cH_0$ will not be rejected if the empirical value
\begin{equation*}
\frac{n}{4 \wh{\theta}_{\!p,\, n}^{\, 2}} \left( \dn - 2 \right)^2 \leq z_{a},
\end{equation*}
and will be rejected otherwise.

\begin{rem}
In the particular case where $\theta_{\!p}^{*} = 0$, the test statistic do not respond under $\cH_1$ as described above. To avoid such situation, we suggest to make use of Theorem \ref{P1_Thm_TlcTheta} for testing beforehand whether $\wh{\theta}_{\!p,\, n}$ is significantly far from zero. Besides, testing $\cH_0\,:\,`` \rho = 0"$ with $\theta_{\!p}^{*} = 0$ amounts to testing the significance of the $p-$th coefficient of the model, not rejected under $\{ \theta_{\!p} \neq 0 \}$. Roughly speaking, under $\{ \theta_{\!p} \neq 0 \} \cap \{ \theta_{\!p}^{*} = 0 \}$, we obviously have $\rho \neq 0$ and the use of Theorem \ref{P3_Thm_TestD} would be irrelevant since $\cH_1$ is certainly true.
\end{rem}

\noindent \textcolor{blue}{As previously mentioned, the statistical procedure of Theorem \ref{P3_Thm_TestD} appears to be a substantial clarification of the \textit{h-test} of Durbin \cite{Durbin70}. To be more precise, formula (12) of \cite{Durbin70} suggests to make use of the test statistic
\begin{equation}
\label{StatH}
\wh{H}_{n} = \rn \, \sqrt{\frac{n}{1 - n \wh{\dV}_{n}(\wh{\theta}_{1,\, n}) }}
\end{equation}
where $\wh{\dV}_{n}(\wh{\theta}_{1,\, n})$ is the least squares estimate of the variance of the first element of $\tn$, and to test it as a standard normal deviate. The presence of an abstract variance estimator not only makes the procedure quite tricky to interpret, but also adds some vulnerability on small-sized samples, as will be observed in the next section. The almost sure equivalence between both test statistics is shown in Appendix D.}

\begin{rem}
\textcolor{blue}{The \textnormal{h-test} of Durbin \cite{Durbin70} is based on the normality assumption on the driven noise $(V_{n})$. As a consequence, $(X_{n})$ is a Gaussian process and the maximum likelihood strategy is suitable not only to provide the estimates, but also to determine their conditional distributions. One can observe that all our results hold without any Gaussianity assumption on $(V_{n})$. Hence, Theorem \ref{P3_Thm_TestD} appears to generalize the \textnormal{h-test} of Durbin.}
\end{rem}

\begin{proof}
The proofs of Theorems \ref{P3_Thm_CvgD} to \ref{P3_Thm_RatD} are left to the reader as they follow essentially the same lines as those given in Appendix C of \cite{BercuProia11}. Theorem \ref{P3_Thm_TestD} is an immediate consequence of Theorem \ref{P3_Thm_TlcD}, noticing that $\sigma^2_{\rho}$ reduces to $\theta_{\!p}^{\, 2}$ under $\cH_0$ and using the same methodology as in the proof of Theorem \ref{P2_Thm_CvgRho}.
\end{proof}

\bigskip

%%%%%%%%%%%%%%%%%%%%%%%%%%%%%%%%%%%%%%%%%%%%%%%%%%%%%%%%%%%%%%%%%%%%%%%%%%%%%%%%

\section{\textcolor{blue}{CONCLUSION}}

%%%%%%%%%%%%%%%%%%%%%%%%%%%%%%%%%%%%%%%%%%%%%%%%%%%%%%%%%%%%%%%%%%%%%%%%%%%%%%%%

\textcolor{blue}{We will now briefly summarize our constatations on simulated samples. Following the same methodology as in Section 5 of \cite{BercuProia11} and also being inspired by the empirical work of Park \cite{Park75}, we have compared the empirical power of the statistical procedure of Theorem \ref{P3_Thm_TestD} with the statistical tests commonly used in time series analysis to detect the presence of a significant first-order correlation in the residuals. Assuming that $\theta_{\! p} \neq 0$ was a statistically significant parameter, our observations were essentially the same as those of \cite{BercuProia11} for different sets of parameters. Namely, on large samples $(n=500)$, we have clearly constated the asymptotic equivalence between the \textit{h-test}, the Breusch-Godfrey test and our statistical procedure, as well as the superiority over the commonly used \textit{portmanteau} tests. On small-sized samples $(n=30)$, our procedure has outperformed all tests by always being more sensitive to the presence of correlation in the residuals, except under $\cH_0$ even if the 84\% of non-rejection were quite satisfying. Our expression of the test statistic seems therefore less vulnerable than the one of Durbin for small sizes. To conclude, the extension of this work to the stable $p-$order autoregressive process where the driven noise is also generated by a $q-$order autoregressive process would constitute a substantial progress in time series analysis. The objective would be to propose a statistical procedure to evaluate $\cH_0\,:\,`` \rho_1 = 0,\, \rho_2 = 0,\, \hdots,\, \rho_{q} = 0 "$ against the alternative hypothesis $\cH_1$ that one can find $1 \leq k \leq q$ such that $\rho_k \neq 0$, based on the Durbin-Watson statistic. In \cite{Durbin70}, Durbin gives an outline of such a strategy which seems rather complicated to implement, relying on power series of infinite orders and under a Gaussianity assumption on the driven noise $(V_{n})$. The author strongly believes that it could be possible to obtain the results explicitly and under weaker assumptions, \textit{via} very tedious calculations. A recent approach in \cite{ButlerPaolella08}, based on saddlepoint approximations for ratios of quadratic forms, could form another way to tackle the problem since the Durbin-Watson statistic is precisely a ratio of quadratic forms.}

\bigskip

\section*{Appendix A}

\begin{center}
{\small ON SOME LINEAR ALGEBRA CALCULATIONS}
\end{center}

\renewcommand{\thesection}{\Alph{section}} 
\renewcommand{\theequation}
{\thesection.\arabic{equation}} \setcounter{section}{1}  
\setcounter{equation}{0}

%%%%%%%%%%%%%%%%%%%%%%%%%%%%%%%%%%%%%%%%%%%%%%%%%%%%%%%%%%%%%%%%%%%%%%%%%%%%%%%%

\subsection*{}
\begin{center}
{\bf A.1. Proof of Lemma \ref{P1_Lem_InvB}.}
\end{center}

%%%%%%%%%%%%%%%%%%%%%%%%%%%%%%%%%%%%%%%%%%%%%%%%%%%%%%%%%%%%%%%%%%%%%%%%%%%%%%%%

We start with the proof of Lemma \ref{P1_Lem_InvB}. Our goal is to show that the matrix $B$ given by \eqref{P1_B} is invertible. Consider the decomposition $B = B_1 + \rho B_2$, where
\begin{equation*}
B_1 =
\begin{pmatrix}
1 & -\theta_1 & -\theta_2 & \hdots & \hdots & \hspace{0.15cm} -\theta_{\!p-1} \hspace{0.15cm} & \hspace{0.15cm} -\theta_{\!p} \hspace{0.15cm} & \hspace{0.15cm} 0 \hspace{0.15cm}\\
\hspace{0.15cm} -\theta_1 \hspace{0.15cm} & 1-\theta_2 & -\theta_3 & \hdots & \hdots & -\theta_{\!p} & 0 & 0\\
-\theta_2 & \hspace{0.15cm} -\theta_1-\theta_3 \hspace{0.15cm} & \hspace{0.15cm} 1-\theta_4 \hspace{0.15cm} & \hdots & \hdots & 0 & 0 & 0\\
\vdots & \vdots & \vdots & & & \vdots & \vdots & \vdots\\
\vdots & \vdots & \vdots & & & \vdots & \vdots & \vdots\\
-\theta_{\!p} & -\theta_{\!p-1} & -\theta_{\!p-2} & \hdots & \hdots & -\theta_1 & 1 & 0\\
0 & -\theta_{\!p} & -\theta_{\!p-1} & \hdots & \hdots & -\theta_2 & -\theta_1 & 1\\
\end{pmatrix},
\end{equation*}
\begin{equation*}
B_2 =
\begin{pmatrix}
0 & -1 & \theta_1 & \hdots & \hspace{0.15cm} \hdots \hspace{0.15cm} & \hspace{0.15cm} \theta_{\!p-2} \hspace{0.15cm} & \hspace{0.15cm} \theta_{\!p-1} \hspace{0.15cm} & \hspace{0.2cm} \theta_{\!p} \hspace{0.15cm}\\
-1  & \theta_1  & \theta_2  &  \hdots & \hdots & \theta_{\!p-1}  & \theta_{\!p}  &  0\\
\theta_1  & -1+\theta_2 & \theta_3 & \hdots & \hdots & \theta_{\!p} & 0 & 0\\
\vdots & \vdots & \vdots & & & \vdots & \vdots & \vdots\\
\vdots & \vdots & \vdots & & & \vdots & \vdots & \vdots\\
\hspace{0.15cm} \theta_{\!p-1} \hspace{0.15cm} & \hspace{0.2cm} \theta_{\!p-2}+ \theta_{\!p} \hspace{0.15cm} & \hspace{0.2cm} \theta_{\!p-3} \hspace{0.2cm} & \hdots & \hdots & -1 & 0 & 0\\
\theta_{\!p} & \theta_{\!p-1} & \theta_{\!p-2} & \hdots & \hdots & \theta_{1} & -1 & 0\\
 \end{pmatrix}.
\end{equation*}

\medskip

\noindent It is trivial to see that $\vert \theta_i + \theta_j \vert \leq \vert \theta_i \vert + \vert \theta_j \vert$ for all $1 \leq i,j \leq p$, and the same goes for $1-\vert \theta_i \vert \leq \vert 1 - \theta_i \vert$. These inequalities immediately imply that $B_1$ is strictly diagonally dominant, and thus invertible by virtue of Levy-Desplanques' theorem 6.1.10 of \cite{HornJohnson90}. Hence, $B = (I_{\! p+2} + \rho B_2 B_1^{-1})B_1$ and the invertibility of $B$ only depends on the spectral radius of $\rho B_2 B_1^{-1}$, \textit{i.e.} the supremum modulus of its eigenvalues. One can explicitly obtain, by a straightforward calculation, that
\begin{equation*}
B_2 B_1^{-1} =
\begin{pmatrix}
\hspace{0.15cm} -\theta_1 \hspace{0.15cm} & \hspace{0.15cm} -1-\theta_2 \hspace{0.15cm} & \hspace{0.15cm} \theta_1-\theta_3 \hspace{0.15cm} & \hdots & \hspace{0.15cm} \theta_{\!p-2}-\theta_{\!p} \hspace{0.15cm} & \hspace{0.15cm} \theta_{\!p-1} \hspace{0.15cm} & \hspace{0.15cm} \theta_{\!p} \hspace{0.15cm}\\
-1 & 0 & \hdots & \hdots & \hdots & \hdots & 0\\
0 & -1 & 0 & \hdots &  \hdots & \hdots & 0\\
\vdots & \ddots & \ddots & \ddots & & & \vdots \\
\vdots & & \ddots & \ddots & \ddots & & \vdots \\
0 & \hdots & \hdots & 0 & -1 & 0 & 0 \\
0 & \hdots & \hdots & \hdots & 0 & -1 & 0\\
\end{pmatrix}.
\end{equation*}

\medskip

\noindent The sum of the first row of $B_2 B_1^{-1}$ is $-1$, involving \textit{de facto} that $-1$ is an eigenvalue of $B_2 B_1^{-1}$ associated with the $(p+2)-$dimensional eigenvector $\begin{pmatrix} 1 & 1 & \hdots & 1 \end{pmatrix}^{\prime}$. By the same way, it is clear that $1$ is an eigenvalue of $B_2 B_1^{-1}$ associated with the eigenvector $\begin{pmatrix} 1 & -1 & \hdots & (-1)^{p+1} \end{pmatrix}^{\prime}$. Let $P(\lambda) = \det(B_2 B_1^{-1} - \lambda I_{\!p+2})$ be the characteristic polynomial of $B_2 B_1^{-1}$. Then, $P(\lambda)$ is recursively computable and explicitly given by
\begin{equation}
\label{A1_InvB_PolyP}
P(\lambda) = (-\lambda)^{p+2} + \sum_{k=1}^{p+2} b_{k}\, (-\lambda)^{p+2-k}
\end{equation}
where $(b_{k})$ designates, for $k \in \{1, \hdots, p+2\}$, the elements of the first line of $B_2 B_1^{-1}$. Since $-1$ and 1 are zeroes of $P(\lambda)$, there exists a polynomial $Q(\lambda)$ of degree $p$ such that $P(\lambda) = (\lambda^2-1)Q(\lambda)$, and a direct calculation shows that $Q$ is given by
 \begin{equation}
\label{A1_InvB_PolyQ}
Q(\lambda) = (-\lambda)^{p} - \sum_{k=1}^{p} \theta_{k}\, (-\lambda)^{p-k}.
\end{equation}
Furthermore, let $R(\lambda)$ be the polynomial of degree $p$ defined as
 \begin{equation}
\label{A1_InvB_PolyR}
R(\lambda) = \lambda^{p} - \sum_{k=1}^{p} \vert\, \theta_{k} \vert\, \lambda^{p-k},
\end{equation}
and note that we clearly have $R(\vert \lambda \vert) \leq \vert Q(\lambda) \vert$, for all $\lambda \in \dC$. Assume that $\lambda_0 \in \dC$ is an eigenvalue of $B_2 B_1^{-1}$ such that $\vert \lambda_0 \vert > 1$. Then,
\begin{eqnarray*}
R(\vert \lambda_0 \vert) & = & \vert \lambda_0 \vert^{p} - \sum_{k=1}^{p} \vert\, \theta_{k} \vert \vert \lambda_0 \vert^{p-k} = \vert \lambda_0 \vert^{p} \left( 1 - \sum_{k=1}^{p} \vert\, \theta_{k} \vert \vert \lambda_0 \vert^{-k} \right), \\
 & \geq & \vert \lambda_0 \vert^{p} \left( 1 - \sum_{k=1}^{p} \vert\, \theta_{k} \vert \right) > 0
\end{eqnarray*}
as soon as $\Vert \theta \Vert_1 < 1$. Consequently, $\vert Q(\lambda_0) \vert > 0$. This obviously contradicts the hypothesis that $\lambda_0$ is an eigenvalue of $B_2 B_1^{-1}$. \textcolor{blue}{This strategy is closely related to the classical result of Cauchy on the location of zeroes of algebraic polynomials, see \textit{e.g.} Theorem 2.1 of \cite{MilovanovicRassias00}.} In conclusion, all the zeroes of $Q(\lambda)$ lie in the unit circle, implying $\rho(B_2 B_1^{-1}) \leq 1$. Since 1 and $-1$ are eigenvalues of $B_2 B_1^{-1}$, we have precisely $\rho(B_2 B_1^{-1}) = 1$, and therefore $\rho(\rho B_2 B_1^{-1}) = \vert \rho \vert < 1$. This guarantees the invertibility of $B$ under the stability conditions, achieving the proof of Lemma \ref{P1_Lem_InvB}. Finally, Corollary \ref{P1_Cor_InvC} immediately follows from Lemma \ref{P1_Lem_InvB}. As a matter of fact, since $B$ is invertible, we have $\det(B) \neq 0$. Denote by $b$ the first diagonal element of $B^{-1}$. Since $\det(C)$ is the cofactor of the first diagonal element of $B$, we have
\begin{equation}
\label{A1_InvC_Cof}
b = \frac{\det(C)}{\det(B)}.
\end{equation}
However, it follows from \eqref{P1_VecLim} that $b = \lambda_0$. We shall prove in the next subsection that the matrix $\Delta_{p}$ given by \eqref{P1_Lambda} is positive definite. It clearly implies that $\lambda_0 > 0$ which means that $b > 0$, so $\det(C) \neq 0$, and the matrix $C$ is invertible.
\hfill
$\mathbin{\vbox{\hrule\hbox{\vrule height1ex \kern.5em\vrule height1ex}\hrule}}$

\bigskip

%%%%%%%%%%%%%%%%%%%%%%%%%%%%%%%%%%%%%%%%%%%%%%%%%%%%%%%%%%%%%%%%%%%%%%%%%%%%%%%%

\subsection*{}
\begin{center}
{\bf A.2. Proof of Lemma \ref{P1_Lem_InvL}.}
\end{center}

%%%%%%%%%%%%%%%%%%%%%%%%%%%%%%%%%%%%%%%%%%%%%%%%%%%%%%%%%%%%%%%%%%%%%%%%%%%%%%%%

Let us start by proving that the spectral radius of the companion matrix associated with model \eqref{Int_Mod} is strictly less than 1. By virtue of the fundamental autoregressive equation \eqref{A11_NewAR} detailed in the next section, the system \eqref{Int_Mod} can be rewritten in the vectorial form, for all $n \geq p+1$,
\begin{equation}
\label{A1_InvL_X}
\Phi_{n}^{p+1} = C_{\! A} \Phi_{n-1}^{p+1} + W_{n}
\end{equation}
where $\Phi_{n}^{p+1}$ stands for the extension of $\Phi_{n}^{p}$ given by \eqref{P1_Phi_Lag} to the next dimension, $W_{n} = \begin{pmatrix} V_{n} & 0 & \hdots & 0 \end{pmatrix}^{\prime}$ and where the companion matrix of order $p+1$
\begin{equation}
\label{A1_InvL_CompMat}
C_{\! A} = \begin{pmatrix}
\theta_1 + \rho ~ & \theta_2 - \theta_1 \rho ~ & \hdots & \theta_{\!p} - \theta_{\!p-1} \rho ~ & -\theta_{\!p} \rho \\
1 & 0 & \hdots & 0 & 0 \\
0 & 1 & \hdots & 0 & 0 \\
\vdots & \vdots & \ddots & \vdots & \vdots \\
0 & 0 & \hdots & 1 & 0
\end{pmatrix}.
\end{equation}

\medskip

\noindent Let $P_{\! A}(\mu) = \det(C_{\! A} - \mu I_{\!p+1})$ be the characteristic polynomial of $C_{\! A}$. Then, it follows from Lemma 4.1.1 of \cite{Duflo97} that
\begin{eqnarray}
P_{\! A}(\mu) & = & (-1)^{p} \left( \mu^{p+1} - (\theta_1 + \rho) \mu^{p} - \sum_{k=2}^{p} \left( \theta_{\!k} - \theta_{k-1} \rho \right) \mu^{p+1-k} + \theta_{\!p} \rho \right), \nonumber \\
\label{A1_InvL_PolyPA}
 & = & (-1)^{p}\,\, ( \mu - \rho ) \left( \mu^{p} - \sum_{k=1}^{p} \theta_{\!k} \mu^{p-k} \right) = (-1)^{p} (\mu - \rho) P(\mu)
\end{eqnarray}
where the polynomial
\begin{equation*}
P(\mu) = \mu^{p} - \sum_{k=1}^{p} \theta_{\!k}\, \mu^{p-k}.
\end{equation*}
Assume that $\mu_0 \in \dC$ is an eigenvalue of $C_{\! A}$ such that $\vert \mu_0 \vert \geq 1$. Then, under the stability condition $\vert \rho \vert < 1$, we obviously have $\mu_0 \neq \rho$. Consequently, we obtain that $P(\mu_0) = 0$ which implies, since $\mu_0 \neq 0$, that
\begin{equation}
\label{A1_InvL_CondVP}
1 - \sum_{k=1}^{p} \theta_{\!k}\, \mu_0^{-k} = 0.
\end{equation}
Nevertheless,
\begin{equation*}
\left\vert \sum_{k=1}^{p} \theta_{\!k}\, \mu_0^{-k} \right\vert \leq \sum_{k=1}^{p} \vert\, \theta_{\!k} \vert \vert \mu_0^{-k} \vert \leq \sum_{k=1}^{p} \vert\, \theta_{\!k} \vert < 1
\end{equation*}
as soon as $\Vert \theta \Vert_1 < 1$ which contradicts \eqref{A1_InvL_CondVP}. Hence, $\rho(C_{\! A}) < 1$ under the stability conditions $\Vert \theta \Vert_1 < 1$ and $\vert \rho \vert < 1$. Hereafter, let $(Y_{n})$ be the stationary autoregressive process satisfying, for all $n \geq p+1$,
\begin{equation}
\label{A1_InvL_Y}
\Psi_{n}^{p+1} = C_{\! A} \Psi_{n-1}^{p+1} + W_{n}
\end{equation}
where
\begin{equation*}  
\Psi^{p+1}_{n} = \begin{pmatrix} Y_{n} & \hspace{0.1cm} & Y_{n-1} & \hspace{0.1cm} & \hdots & \hspace{0.1cm} & Y_{n-p} \end{pmatrix}^{\prime}.
\end{equation*}
It follows from \eqref{A1_InvL_Y} that, for all $n \geq p+1$,
\begin{equation*}
Y_{n} = (\theta_1 + \rho) Y_{n-1} + \sum_{k=2}^{p} (\theta_{\!k} - \theta_{\!k-1} \rho) Y_{n-k} - \theta_{\!p} \rho Y_{n-p-1} + V_{n}.
\end{equation*}
By virtue of Theorem 4.4.2 of \cite{BrockwellDavis91}, the spectral density of the process $(Y_{n})$ is given, for all $x$ in the torus $ \dT = [-\pi, \pi]$, by
\begin{equation}
\label{A1_InvL_SpecDens}
f_{Y}(x) = \frac{\sigma^2}{2 \pi \vert A(e^{- \mathrm{i} x}) \vert^{\, 2}}
\end{equation}
where the polynomial $A$ is defined, for all $\mu \neq 0$, as
\begin{equation}
\label{A1_InvL_PolyA}
A(\mu) = (-1)^{p}\, \mu^{p+1} P_{\! A}(\mu^{-1}),
\end{equation}
in which $P_{A}$ is the polynomial given in \eqref{A1_InvL_PolyPA}, and $A(0) = 1$. In light of foregoing, $A$ has no zero on the unit circle. In addition, for all $k \in \dZ$, denote by
\begin{equation*}
\wh{f}_{k} = \int_{\dT} f_{Y}(x) e^{- \mathrm{i} k x} \, \mathrm{d}x
\end{equation*}
the Fourier coefficient of order $k$ associated with $f_{Y}$. It is well-known that, for all $p \geq 1$, the covariance matrix of the vector $\Psi_{n}^{p}$ coincides with the Toeplitz matrix of order $p$ of the spectral density $f_{Y}$ in \eqref{A1_InvL_SpecDens}. More precisely, for all $p \geq 1$, we have
\begin{equation}
\label{A1_InvL_ToepCov}
T_{\!p}(f_{Y}) = \left( \wh{f}_{i-j} \right)_{1\, \leq\,\, i,\, j\, \leq\, p} = \sigma^2 \Delta_{p}
\end{equation}
where $\Delta_{p}$ is given by \eqref{P1_Lambda} and $T$ stands for the Toeplitz operator. As a matter of fact, since $\rho(C_{\! A}) < 1$, we have
\begin{equation*}
\lim_{n\rightarrow \infty} \dE\Big[ \Phi_{n}^{p} {\Phi_{n}^{p}}^{\: \prime} \Big] = \dE\Big[ \Psi_{p}^{p} \Psi_{p}^{p \:\, \prime} \Big] = \sigma^2 \Delta_{p}.
\end{equation*}
Finally, we deduce from Proposition 4.5.3 of \cite{BrockwellDavis91}, \textcolor{blue}{or from the properties of Toeplitz operators deeply studied in \cite{GrenanderSzego58},} that
\begin{equation}
\label{A1_InvL_MinVP}
2\pi m_{f} \leq \lambda_{\min}(T_{\!p}(f_{Y})) \leq \lambda_{\max}(T_{\!p}(f_{Y})) \leq 2 \pi M_{f}
\end{equation}
where
\begin{equation*}
m_{f} = \min_{x \, \in \, \dT} f_{Y}(x) \hspace{0.5cm} \text{and} \hspace{0.5cm} M_{f} = \max_{x \,\in \, \dT} f_{Y}(x).
\end{equation*}
Therefore, as $m_{f} > 0$, $T_{\!p}(f_{Y})$ is positive definite, which clearly ensures that for all $p \geq 1$, $\Delta_{p}$ is also positive definite. This achieves the proof of Lemma \ref{P1_Lem_InvL}.
\hfill
$\mathbin{\vbox{\hrule\hbox{\vrule height1ex \kern.5em\vrule height1ex}\hrule}}$

%%%%%%%%%%%%%%%%%%%%%%%%%%%%%%%%%%%%%%%%%%%%%%%%%%%%%%%%%%%%%%%%%%%%%%%%%%%%%%%%

\bigskip

\section*{Appendix B}

\begin{center}
{\small PROOFS OF THE AUTOREGRESSIVE PARAMETER RESULTS}
\end{center}

\renewcommand{\thesection}{\Alph{section}} 
\renewcommand{\theequation}
{\thesection.\arabic{equation}} \setcounter{section}{2}  
\setcounter{equation}{0}

%%%%%%%%%%%%%%%%%%%%%%%%%%%%%%%%%%%%%%%%%%%%%%%%%%%%%%%%%%%%%%%%%%%%%%%%%%%%%%%%

\subsection*{}
\begin{center}
{\bf B.1. Preliminary Lemmas.}
\end{center}

%%%%%%%%%%%%%%%%%%%%%%%%%%%%%%%%%%%%%%%%%%%%%%%%%%%%%%%%%%%%%%%%%%%%%%%%%%%%%%%%

We start with some useful technical lemmas we shall make repeatedly use of. The proof of Lemma \ref{A1_Lem_StabV} may be found in the one of Corollary 1.3.21 in \cite{Duflo97}.

\begin{lem}
\label{A1_Lem_StabV}
Assume that $(V_{n})$ is a sequence of independent and identically distributed random variables such that, for some $a \geq 1$, $\dE[|V_1|^a]$ is finite. Then,
\begin{equation}
\label{A1_Lem_StabV_Lgn}
\lim_{n\rightarrow \infty} \frac{1}{n} \sum_{k=1}^{n} \vert V_k \vert^a = \dE[|V_1|^a] \textnormal{\cvgps}
\end{equation}
and
\begin{equation}
\label{A1_Lem_StabV_Sup}
\sup_{1 \leq k \leq n}  \vert V_k \vert = o(n^{1/a}) \textnormal{\cvgps}
\end{equation}
\end{lem}

\begin{lem}
\label{A1_Lem_StabX}
Assume that $(V_{n})$ is a sequence of independent and identically distributed random variables such that, for some $a \geq 1$, $\dE[|V_1|^a]$ is finite. If $(X_n)$ satisfies \eqref{Int_Mod} with $\Vert \theta \Vert_1 < 1$ and $\vert \rho \vert < 1$, then
\begin{equation}
\label{A1_Lem_StabX_Lgn}
\sum_{k=0}^{n} \vert X_k \vert^a = O(n) \textnormal{\cvgps}
\end{equation}
and
\begin{equation}
\label{A1_Lem_StabX_Sup}
\sup_{0 \leq k \leq n}  \vert X_k \vert = o(n^{1/a}) \textnormal{\cvgps}
\end{equation}
\end{lem}

\begin{rem}
In the particular case where $a=4$, we obtain that
\begin{equation*}
\sum_{k=0}^{n} X_k^4=O(n) \textnormal{\cvgps} \hspace{1cm} \text{and} \hspace{1cm} \sup_{0 \leq k \leq n} X_k^2 = o(\sqrt{n}) \textnormal{\cvgps}
\end{equation*}
\end{rem}

\begin{proof}
The reader may find an approach following essentially the same lines in the proof of Lemma A.2 in \cite{BercuProia11}, merely considering the stability condition $\Vert \theta \Vert_1 < 1$ in lieu of $\vert \theta \vert < 1$.
\end{proof}

\begin{lem}
\label{A1_Lem_LimSn}
Assume that the initial values $X_0, X_1, \hdots, X_{p-1}$ with $\veps_0=X_0$ are square-integrable and that $(V_{n})$ is a sequence of independent and identically distributed random variables with zero mean and variance $\sigma^2 > 0$. Then, under the stability conditions $\Vert \theta \Vert_1 < 1$ and $\vert \rho \vert < 1$, we have the almost sure convergence
\begin{equation}
\label{A12_LimSn}
\lim_{n\rightarrow \infty} \frac{S_{n}}{n} = \sigma^2 \Delta_{p} \textnormal{\cvgps}
\end{equation}
where the matrix $\Delta_{p}$ is given by \eqref{P1_Lambda}.
\end{lem}

\begin{proof}
By adopting the same approach as the one used to prove Theorem 2.2 in \cite{BercuProia11}, it follows from the fundamental autoregressive equation \eqref{A11_NewAR}, that will be detailed in the next section, that for all $0 \leq d \leq p+1$,
\begin{equation*}
\lim_{n\rightarrow \infty} \frac{1}{n} \sum_{k=1}^{n} X_{k-d} V_{k} = \sigma^2 \delta_{d} \textnormal{\cvgps}
\end{equation*}
where $\delta_{d}$ stands for the Kronecker delta function equal to 1 when $d=0$, and 0 otherwise. Denote by $\ell_{d}$ the limiting value which verifies, by virtue of Lemma \ref{A1_Lem_StabX} together with Corollary 1.3.25 of \cite{Duflo97},
\begin{equation*}
\lim_{n\rightarrow \infty} \frac{1}{n} \sum_{k=1}^{n} X_{k-d} X_{k} = \ell_{d} \textnormal{\cvgps}
\end{equation*}
Finally, let also $L \in \dR^{p+2}$ and, for $0 \leq d \leq p+1$, $L_{p}^{d} \in \dR^{p}$ be vectors of limiting values such that,
\begin{equation*}
L = \begin{pmatrix}
\ell_{0} & \ell_{1} & \hdots & \ell_{p+1}
\end{pmatrix}^{\prime} \hspace{0.5cm} \text{and} \hspace{0.5cm}
L_{p}^{d} = \begin{pmatrix}
\ell_{d} & \ell_{d-1} & \hdots & \ell_{d-p+1}
\end{pmatrix}^{\prime}.
\end{equation*}
From \eqref{A11_NewAR}, an immediate development leads to
\begin{equation*}
\sum_{k=1}^{n} X_{k-d} X_{k} = \beta^{\prime} \sum_{k=1}^{n} \Phi_{k-1}^p X_{k-d} - \theta_{\!p} \rho \sum_{k=1}^{n} X_{k-p-1} X_{k-d} + \sum_{k=1}^{n} X_{k-d} V_{k},
\end{equation*}
considering that $X_{-1}, X_{-2}, \hdots, X_{-p} = 0$. Consequently, we obtain a set of relations between almost sure limits, for all $0 \leq d \leq p+1$,
\begin{equation}
\label{A12_SysLinLim}
\ell_{d} = \beta^{\, \prime} L_{p}^{d-1} - \theta_{\!p} \rho \ell_{d-p-1} + \sigma^2 \delta_{d}
\end{equation}
where $\ell_{-d} = \ell_{d}$. Hereafter, if $d$ varies from 0 to $p+1$, one can build a $(p+2) \times (p+2)$ linear system of equations verifying
\begin{equation}
\label{A12_SysLinMat}
B L = \sigma^2 e
\end{equation}
where $B$ is precisely given by \eqref{P1_B}. We know from Lemma \ref{P1_Lem_InvB} that under the stability conditions, the matrix $B$ is invertible. Therefore, it follows that
\begin{equation*}
L = \sigma^2 B^{-1} e,
\end{equation*}
meaning \textit{via} \eqref{P1_VecLim} that $L = \sigma^2 \Lambda$, or else, for all $0 \leq d \leq p+1$, $\ell_{d} = \sigma^2 \lambda_{d}$, which completes the proof of Lemma \ref{A1_Lem_LimSn}.
\end{proof}

%%%%%%%%%%%%%%%%%%%%%%%%%%%%%%%%%%%%%%%%%%%%%%%%%%%%%%%%%%%%%%%%%%%%%%%%%%%%%%%%

\subsection*{}
\begin{center}
{\bf B.2. Proof of Theorem \ref{P1_Thm_CvgTheta}.}
\end{center}

%%%%%%%%%%%%%%%%%%%%%%%%%%%%%%%%%%%%%%%%%%%%%%%%%%%%%%%%%%%%%%%%%%%%%%%%%%%%%%%%

We easily deduce from \eqref{Int_Mod} that the process $(X_n)$ satisfies the fundamental autoregressive equation given, for all $n\geq p+1$, by
\begin{equation}
\label{A11_NewAR}
X_{n} = \beta^{\, \prime} \Phi_{n-1}^p - \theta_{\!p} \rho X_{n-p-1} + V_{n}
\end{equation}
where $\beta$ is given by \eqref{P1_Beta}. On the basis of \eqref{A11_NewAR}, consider the summation
\begin{equation}
\label{A11_InitSum}
\sum_{k=1}^{n} \Phi_{k-1}^p X_{k} = \sum_{k=1}^{n} \Phi_{k-1}^p \beta^{\, \prime} \Phi_{k-1}^p - \theta_{\!p} \rho \sum_{k=1}^{n} \Phi_{k-1}^p X_{k-p-1} + \sum_{k=1}^{n} \Phi_{k-1}^p V_{k}.
\end{equation}

\medskip

\noindent First of all, an immediate calculation leads to
\begin{equation}
\label{A11_T1}
\sum_{k=1}^{n} \Phi_{k-1}^p \beta^{\, \prime} \Phi_{k-1}^p = (S_{n-1} - S) \beta
\end{equation}
where $S_{n-1}$ and $S$ are given in \eqref{P1_Sn}. Let us focus now on the more intricate term
\begin{equation*}
\sum_{k=1}^{n} \Phi_{k-1}^p X_{k-p-1}
\end{equation*}
in which we shall expand each element of $\Phi_{k-1}^p$ according to \eqref{A11_NewAR}. A direct calculation infers the equality, for all $n \geq p+1$,
\begin{equation}
\label{A11_T2}
\sum_{k=1}^{n} \Phi_{k-1}^p X_{k-p-1} = S_{n-1}\, J_{\! p}\, \beta - \theta_{\!p} \rho \sum_{k=1}^{n} \Phi_{k-1}^p X_{k} + J_{\! p} \sum_{k=1}^{n} \Phi_{k-1}^p V_{k} + \xi_{n}
\end{equation}
where Lemma \ref{A1_Lem_StabX} ensures that the remainder term $\xi_{n}$ is made of isolated terms such that $\Vert \xi_{n} \Vert = o(n)$ a.s. Let also $M_{n}$ be the $p-$dimensional martingale
\begin{equation}
\label{A11_Mn}
M_{n} = \sum_{k=1}^{n} \Phi_{k-1}^p V_{k}.
\end{equation}
We deduce from \eqref{A11_InitSum} together with \eqref{A11_T1} and \eqref{A11_T2} that
\begin{equation*}
\sum_{k=1}^{n} \Phi_{k-1}^p X_{k} = \alpha S_{n-1} (I_{\! p} - \theta_{\!p} \rho J_{\! p}) \beta + \alpha (I_{\! p} - \theta_{\!p} \rho J_{\! p}) M_{n} + \alpha \xi_{n}
\end{equation*}
where $\alpha$ is given by \eqref{P1_Alpha}. Thus, taking into account the expression of the estimator \eqref{P1_Est}, we get the main decomposition, for all $n \geq p+1$,
\begin{equation}
\label{A11_DecompEst}
\tn = \alpha (I_{\! p} - \theta_{\!p} \rho J_{\! p}) \beta + \alpha (S_{n-1})^{\! -1} (I_{\! p} - \theta_{\!p} \rho J_{\! p}) M_{n} + \alpha (S_{n-1})^{\! -1} \xi_{n}.
\end{equation}

\medskip
 
\noindent For all $n \geq 1$, denote by $\cF_{n}$ the $\sigma-$algebra of the events occurring up to stage $n$, $\cF_{n} = \sigma(X_0, \hdots, X_{p}, V_1, \hdots, V_n)$. The random sequence $(M_n)$ given by \eqref{A11_Mn} is a locally square-integrable real vector martingale \cite{Duflo97}, \cite{HallHeyde80}, adapted to $\cF_{n}$, with predictable quadratic variation given, for all $n \geq 1$, by
\begin{eqnarray}
\langle M \rangle_n & = & \sum_{k=1}^n \dE[(\Delta M_k) (\Delta M_k)^{\prime}|\cF_{k-1}], \nonumber\\
\label{A11_Proc}
 & = & \sigma^2 \sum_{k=1}^{n} \Phi^p_{k-1} \Phi^{p \,\, {\prime}}_{k-1} = \sigma^2 (S_{n-1} - S)
\end{eqnarray}
where $\Delta M_{k}$ stands for the difference $M_{k} - M_{k-1}$. We know from Lemma \ref{A1_Lem_LimSn} that
\begin{equation}
\label{A11_LimProc}
\lim_{n\rightarrow \infty} \frac{S_{n}}{n} = \sigma^2 \Delta_{p} \textnormal{\cvgps}
\end{equation}
and $\Delta_{p}$ is positive definite as a result of Lemma \ref{P1_Lem_InvL}. Then, \eqref{A11_LimProc} implies that
\begin{equation}
\label{A11_CondVP1}
\lim_{n\rightarrow \infty} \frac{\text{tr}(S_{n})}{n} = \sigma^2 p\, \lambda_0 \textnormal{\cvgps}
\end{equation}
where $\lambda_0 > 0$. Moreover, since $\Delta_{p}$ is positive definite, we also have that
\begin{equation}
\label{A11_CondVP2}
\lambda_{\text{max}}(S_{n}) = O\left(\lambda_{\text{min}}(S_{n}) \right) \textnormal{\cvgps}
\end{equation}
Consequently, we deduce from \eqref{A11_Proc}, \eqref{A11_CondVP1}, \eqref{A11_CondVP2} and the strong law of large numbers for vector martingales given \textit{e.g.} in Theorem 4.3.15 of \cite{Duflo97}, or \cite{DufloSenoussiTouati90} that,
\begin{equation}
\label{A11_LGN}
\lim_{n\rightarrow \infty} \langle M \rangle_{n}^{-1} M_{n} = 0 \textnormal{\cvgps}
\end{equation}
and obviously,
\begin{equation}
\label{A11_LGNComb}
\lim_{n\rightarrow \infty} (S_{n-1})^{\! -1} (I_{\! p} - \theta_{\!p} \rho J_{\! p}) M_{n} = 0 \textnormal{\cvgps}
\end{equation}
As mentioned above, $(V_{n})$ having a finite moment of order 2 implies, \textit{via} Lemma \ref{A1_Lem_StabX} and \eqref{A11_LimProc}, that
\begin{equation}
\label{A11_LGNReste}
\lim_{n\rightarrow \infty} (S_{n-1})^{\! -1} \xi_{n} = 0 \textnormal{\cvgps}
\end{equation}
Finally, \eqref{A11_DecompEst} together with \eqref{A11_LGNComb} and \eqref{A11_LGNReste} achieve the proof of Theorem \ref{P1_Thm_CvgTheta},
\begin{equation*}
\lim_{n\rightarrow \infty} \tn = \alpha (I_{\! p} - \theta_{\!p} \rho J_{\! p}) \beta \textnormal{\cvgps}
\end{equation*}
\hfill
$\mathbin{\vbox{\hrule\hbox{\vrule height1ex \kern.5em\vrule height1ex}\hrule}}$

%%%%%%%%%%%%%%%%%%%%%%%%%%%%%%%%%%%%%%%%%%%%%%%%%%%%%%%%%%%%%%%%%%%%%%%%%%%%%%%%

\subsection*{}
\begin{center}
{\bf B.3. Proof of Theorem \ref{P1_Thm_TlcTheta}.}
\end{center}

%%%%%%%%%%%%%%%%%%%%%%%%%%%%%%%%%%%%%%%%%%%%%%%%%%%%%%%%%%%%%%%%%%%%%%%%%%%%%%%%

The main decomposition \eqref{A11_DecompEst} enables us to write, for all $n \geq p+1$,
\begin{equation}
\label{A12_DecompDiff}
\sqrt{n} \left( \tn - \theta^{*} \right) = \alpha \sqrt{n}\, (S_{n-1})^{\! -1} (I_{\! p} - \theta_{\!p} \rho J_{\! p}) M_{n} + \alpha \sqrt{n}\, (S_{n-1})^{\! -1}\, \xi_{n}.
\end{equation}

\medskip

\noindent On the one hand, we have from Lemma \ref{A1_Lem_StabX} with $a=4$ that $\Vert \xi_{n} \Vert = o(\sqrt{n})$ a.s. assuming the existence of a finite moment of order 4 for $(V_{n})$. Hence, \textit{via} \eqref{A11_LimProc},
\begin{equation}
\label{A12_TLCReste}
\lim_{n\rightarrow \infty} \sqrt{n}\, (S_{n-1})^{\! -1}\, \xi_{n} = 0 \textnormal{\cvgps}
\end{equation}
On the other hand, we shall make use of the central limit theorem for vector martingales given \textit{e.g.} by Corollary 2.1.10 of \cite{Duflo97}, to establish the asymptotic normality of the first term in the right-hand side of \eqref{A12_DecompDiff}.  Foremost, it is necessary to prove that the Lindeberg's condition is satisfied. We have to prove that, for all $\veps > 0$,
\begin{equation}
\label{A12_LindCond}
\frac{1}{n} \sum_{k=1}^{n} \dE \left[ \Vert \Delta M_{k} \Vert^2 ~ \rI_{\left\{ \Vert \Delta M_{k} \Vert\, \geq\, \veps \sqrt{n} \right\}} \vert \cF_{k-1} \right] \limp 0
\end{equation}
where $\Delta M_{k} = M_{k} - M_{k-1} = \Phi_{k-1}^{p} V_{k}$. We have from Lemma \ref{A1_Lem_StabX} with $a=4$ that
\begin{equation}
\label{A12_SommePhi4}
\sum_{k=1}^{n} \Vert \Phi_{k-1}^p \Vert^4 = O(n) \textnormal{\cvgps}
\end{equation}
Moreover, for all $\veps > 0$,
\begin{eqnarray*}
\frac{1}{n} \sum_{k=1}^{n} \dE \left[ \Vert \Delta M_{k} \Vert^2 ~ \rI_{\left\{ \Vert \Delta M_{k} \Vert\, \geq\, \veps \sqrt{n} \right\} } \vert \cF_{k-1} \right] & \leq & \frac{1}{\veps^2 n^2} \sum_{k=1}^{n} \dE \left[ \Vert \Delta M_{k} \Vert^4 \vert \cF_{k-1} \right], \\
 & \leq & \frac{\tau^4}{\veps^2\, n^2} \sum_{k=1}^{n} \Vert \Phi_{k-1}^p \Vert^4
\end{eqnarray*}
where $\tau^4$ stands for the moment of order 4 associated with $(V_{n})$. Consequently, \eqref{A12_SommePhi4} ensures that
\begin{equation*}
\frac{1}{n} \sum_{k=1}^{n} \dE \left[ \Vert \Delta M_{k} \Vert^2 ~ \rI_{\left\{ \Vert \Delta M_{k} \Vert\, \geq\, \veps \sqrt{n} \right\} } \vert \cF_{k-1} \right] = O\left( n^{-1} \right) \textnormal{\cvgps}
\end{equation*}
and the Lindeberg's condition \eqref{A12_LindCond} is satisfied. We conclude from the central limit theorem for vector martingales together with Lemma \ref{P1_Lem_InvL} and Lemma \ref{A1_Lem_LimSn} that
\begin{equation}
\label{A12_TLC}
\sqrt{n}\, \langle M \rangle_{n}^{-1} M_{n} \liml \cN\left( 0,\sigma^{-4} \Delta_{p}^{-1} \right)
\end{equation}
where $\Delta_{p}$ is given by \eqref{P1_Lambda}, which leads to
\begin{equation}
\label{A12_TLCComb}
\alpha \sqrt{n}\, (S_{n-1})^{-1} \left(I_{\! p} - \theta_{\!p} \rho J_{\! p} \right) M_{n} \liml \cN\left( 0, \Sigma_{\theta} \right).
\end{equation}
Finally, \eqref{A12_DecompDiff}, \eqref{A12_TLCReste} and \eqref{A12_TLCComb} complete the proof of Theorem \ref{P1_Thm_TlcTheta}.
\hfill
$\mathbin{\vbox{\hrule\hbox{\vrule height1ex \kern.5em\vrule height1ex}\hrule}}$

%%%%%%%%%%%%%%%%%%%%%%%%%%%%%%%%%%%%%%%%%%%%%%%%%%%%%%%%%%%%%%%%%%%%%%%%%%%%%%%%

\subsection*{}
\begin{center}
{\bf B.4. Proof of Theorem \ref{P1_Thm_RatTheta}.}
\end{center}

%%%%%%%%%%%%%%%%%%%%%%%%%%%%%%%%%%%%%%%%%%%%%%%%%%%%%%%%%%%%%%%%%%%%%%%%%%%%%%%%

Let $(W_{n})$ be the sequence of standardization matrices defined as $W_{n} = \sqrt{n}\, I_{\! p}$. Consider the locally square-integrable real vector martingale $(M_{n})$ with predictable quadratic variation $\langle M \rangle_{n}$ given by \eqref{A11_Proc}. Via Lemma \ref{A1_Lem_LimSn}, we have the almost sure convergence
\begin{equation}
\label{A13_H1}
\lim_{n\rightarrow \infty} W_{n}^{-1}\, \langle M \rangle_{n}\, W_{n}^{-1} = \sigma^4 \Delta_{p} \textnormal{\cvgps}
\end{equation}
where $\Delta_{p}$ is given by \eqref{P1_Lambda}. For all $n \geq 0$, denote
\begin{equation}
\label{A13_Tn}
T_{n} = \sum_{k=1}^{n} X_{k}^4
\end{equation}
with $T_0=0$. From Lemma \ref{A1_Lem_StabX} with $a=4$, we have that $T_{n} = O(n)$ a.s. Thus,
\begin{eqnarray*}
\sum_{n=1}^{\infty} \frac{X_{n}^4}{n^2} & = & \sum_{n=1}^{\infty} \frac{T_{n} - T_{n-1}}{n^2} = \sum_{n=1}^{\infty} \left( \frac{2n + 1}{n^2\, (n+1)^2} \right) T_{n}, \\
 & = & O\left( \sum_{n=1}^{\infty} \frac{T_{n}}{n^3} \right) = O\left( \sum_{n=1}^{\infty} \frac{1}{n^2} \right) < +\infty \textnormal{\cvgps}
\end{eqnarray*}
which immediately implies that
\begin{equation}
\label{A13_H3}
\sum_{n=1}^{\infty} \frac{\Vert \Phi_{n-1}^p \Vert^4}{n^2} < +\infty \textnormal{\cvgps}
\end{equation}
From \eqref{A13_H1} and \eqref{A13_H3}, we can deduce that $(M_{n})$ satisfies the quadratic strong law for vector martingales given \textit{e.g.} by Theorem 2.1 of \cite{ChaabaneMaaouia00},
\begin{equation}
\label{A13_LFQ}
\lim_{n\rightarrow \infty} \frac{1}{\log n^p} \sum_{k=1}^{n} \left[1 - \frac{k^p}{(k+1)^p} \right] W_{k}^{-1} M_{k} M_{k}^{\, \prime}\, W_{k}^{-1} = \sigma^4 \Delta_{p} \textnormal{\cvgps}
\end{equation}
Hereafter, it follows from \eqref{A11_DecompEst} that, for all $n \geq p+1$,
\begin{eqnarray}
\label{A13_QuadForm}
\left( \tn - \theta^{*} \right) \left( \tn - \theta^{*} \right)^{\prime} & = & \alpha^2 (S_{n-1})^{\! -1} \Big[ K_{\! p}\, M_{n} + \xi_{n} \Big] \Big[ M_{n}^{\, \prime}\, K_{\! p}  + \xi_{n}^{\prime} \Big] (S_{n-1})^{\! -1}, \nonumber\\
& = & \alpha^2 (S_{n-1})^{\! -1}\, K_{\! p}\, M_{n}\, M_{n}^{\, \prime}\, K_{\! p}\, (S_{n-1})^{\! -1} + \zeta_{n}
\end{eqnarray}
where $K_{\! p} = (I_{\! p} - \theta_{\!p} \rho J_{\! p})$ and the remainder term
\begin{equation*}
\zeta_{n} = \alpha^2 (S_{n-1})^{\! -1} (\xi_{n}\, M_{n}^{\, \prime}\, K_{\! p} + K_{\! p}\, M_{n}\, \xi_{n}^{\, \prime} + \xi_{n}\, \xi_{n}^{\, \prime}) (S_{n-1})^{\! -1}.
\end{equation*}
However, we have from Lemma \ref{P1_Lem_InvL} and Lemma \ref{A1_Lem_LimSn} that
\begin{equation}
\label{A13_LimInvSn}
\lim_{n\rightarrow \infty} n(S_{n-1})^{\! -1} = \sigma^{-2} \Delta_{p}^{-1} \textnormal{\cvgps}
\end{equation}
As a result, \eqref{A13_LFQ}, \eqref{A13_LimInvSn} and a set of additional steps of calculation lead to the almost sure convergence
\begin{equation}
\label{A13_LFQComb}
\lim_{n\rightarrow \infty} \frac{1}{\log n} \sum_{k=1}^{n} (S_{k-1})^{\! -1}\, K_{\! p}\, M_{k}\, M_{k}^{\, \prime}\, K_{\! p}\, (S_{k-1})^{\! -1} = K_{\! p}\, \Delta_{p}^{-1} K_{\! p} \textnormal{\cvgps}
\end{equation}
since $K_{\! p}\, \Delta_{p}^{-1} = \Delta_{p}^{-1} K_{\! p}$ due to the bisymmetry of $\Delta_{p}^{-1}$. Assuming a finite moment of order 4 for $(V_{n})$, one can easily be convinced that $\zeta_{n}$ is going to play a negligible role compared to the first one in the right-hand side of \eqref{A13_QuadForm}. Indeed, we clearly have that $\Vert M_{n} \Vert \Vert \xi_{n} \Vert = o(n^{3/4} \sqrt{\log{n}})$ a.s. It follows that
\begin{equation}
\label{A13_LFQReste}
\sum_{k=1}^{n} \zeta_{k} = O(1) \textnormal{\cvgps}
\end{equation}
Finally, \eqref{A13_LFQComb} and \eqref{A13_LFQReste} complete the proof of the first part of Theorem \ref{P1_Thm_RatTheta},
\begin{equation*}
\lim_{n\rightarrow \infty} \frac{1}{\log n} \sum_{k=1}^{n} \left( \tk - \theta^{*} \right) \left( \tk - \theta^{*} \right)^{\prime} = \Sigma_{\theta} \textnormal{\cvgps}
\end{equation*}
since $\Sigma_{\theta} = \alpha^2 K_{\! p}\, \Delta_{p}^{-1} K_{\! p}$.

\medskip

\noindent The law of iterated logarithm \eqref{P1_LliTheta} is much more easy to handle. It is based on the law of iterated logarithm for vector martingales given \textit{e.g.} by Lemma C.2 in \cite{Bercu98}. Under the assumption \eqref{A13_H3} already verified, for any vector $v \in \dR^p$, we have
\begin{eqnarray}
\label{A13_LLI}
\limsup_{n \rightarrow \infty} \left( \frac{n}{2 \log \log n} \right)^{\! 1/2} v^{\, \prime} (S_{n-1})^{\! -1} M_{n} & = & -\liminf_{n \rightarrow \infty} \left( \frac{n}{2 \log \log n} \right)^{\! 1/2} v^{\, \prime} (S_{n-1})^{\! -1} M_{n}, \nonumber\\
 & = & \sqrt{ v^{\, \prime} \Delta_{p}^{-1} v } \textnormal{\cvgps}
\end{eqnarray}
Via \eqref{A13_LLI} and the negligibility of $\zeta_{n}$, we immediately get
\begin{eqnarray}
\label{A13_LLIComb}
\limsup_{n \rightarrow \infty} \left( \frac{n}{2 \log \log n} \right)^{\! 1/2} v^{\, \prime} \left( \tn - \theta^{*} \right) & = & -\liminf_{n \rightarrow \infty} \left( \frac{n}{2 \log \log n} \right)^{\! 1/2} v^{\, \prime} \left( \tn - \theta^{*} \right), \nonumber\\
 & = & \alpha \sqrt{ v^{\, \prime} K_{\! p}\, \Delta_{p}^{-1} K_{\! p}\, v } \textnormal{\cvgps}
\end{eqnarray}
Since \eqref{A13_LLIComb} is true whatever the value of $v \in \dR^{p}$, we obtain a matrix formulation of the law of iterated logarithm,
\begin{equation}
\label{A13_LLIQuadForm}
\limsup_{n \rightarrow \infty} \left( \frac{n}{2 \log \log n} \right) \left( \tn - \theta^{*} \right) \left( \tn - \theta^{*} \right)^{\prime} = \Sigma_{\theta} \textnormal{\cvgps}
\end{equation}
Passing through the trace in \eqref{A13_LLIQuadForm}, we find that
\begin{equation}
\label{A13_LLIEncadrement}
\limsup_{n \rightarrow \infty} \left( \frac{n}{2 \log \log n} \right) \big\Vert \tn - \theta^{*} \big\Vert^2 = \text{tr}(\Sigma_{\theta}) \textnormal{\cvgps}
\end{equation}
which completes the proof of Theorem \ref{P1_Thm_RatTheta}.
\hfill
$\mathbin{\vbox{\hrule\hbox{\vrule height1ex \kern.5em\vrule height1ex}\hrule}}$

%%%%%%%%%%%%%%%%%%%%%%%%%%%%%%%%%%%%%%%%%%%%%%%%%%%%%%%%%%%%%%%%%%%%%%%%%%%%%%%%

\bigskip

\section*{Appendix C}

\begin{center}
{\small PROOFS OF THE SERIAL CORRELATION PARAMETER RESULTS}
\end{center}

\renewcommand{\thesection}{\Alph{section}} 
\renewcommand{\theequation}
{\thesection.\arabic{equation}} \setcounter{section}{3}  
\setcounter{equation}{0}
\setcounter{lem}{0}

%%%%%%%%%%%%%%%%%%%%%%%%%%%%%%%%%%%%%%%%%%%%%%%%%%%%%%%%%%%%%%%%%%%%%%%%%%%%%%%%

\subsection*{}
\begin{center}
{\bf C.1. Proof of Theorem \ref{P2_Thm_CvgRho}.}
\end{center}

%%%%%%%%%%%%%%%%%%%%%%%%%%%%%%%%%%%%%%%%%%%%%%%%%%%%%%%%%%%%%%%%%%%%%%%%%%%%%%%%

Let us introduce some additional notations to make this technical proof more understandable. Recall that, for all $d \in \{0, \hdots, p+1\}$, we have the almost sure convergence
\begin{equation}
\label{A21_Lim}
\lim_{n\rightarrow \infty} \frac{1}{n} \sum_{k=1}^{n} X_{k-d} X_{k} = \sigma^2 \lambda_{d} \textnormal{\cvgps}
\end{equation}
Let $\Lambda_{p}^{0}$\,, $\Lambda_{p}^{1}$ and $\Lambda_{p}^{2}$ be a set of $p-$dimensional vectors of limiting values such that, for $d = \{0, 1, 2\}$,
\begin{equation}
\label{A21_L012}
\Lambda_{p}^{d} = \begin{pmatrix}
\lambda_{d} & \lambda_{d+1} & \hdots & \lambda_{d+p-1}
\end{pmatrix}^{\prime},
\end{equation}
and note that the almost sure convergence follows,
\begin{equation}
\label{A21_Lim012}
\lim_{n\rightarrow \infty} \frac{1}{n} \sum_{k=1}^{n} \Phi_{k-d}^p X_{k} = \sigma^2 \Lambda_{p}^{d} \textnormal{\cvgps}
\end{equation}
For all $n \geq 1$, denote by $A_{n}$ the square matrix of order $p$ defined as
\begin{equation}
\label{A21_Pn}
A_{n} = \sum_{k=1}^{n} \Phi^p_{k}\, \Phi^{p \: \prime}_{k-1}.
\end{equation}
Following a reasoning very similar to the proof of Theorem \ref{P1_Thm_CvgTheta}, it is possible to obtain the decomposition, for all $n \geq p+1$,
\begin{equation}
\label{A21_DevL0}
\sum_{k=1}^{n} \Phi_{k}^p X_{k} = A_{n}\, \theta^{*} + \alpha \sum_{k=1}^{n} \Phi_{k}^{p}\, V_{k} - \alpha\, \theta_{\!p} \rho\, J_{\! p} \sum_{k=1}^{n} \Phi_{k-2}^p V_{k} + \eta_{n}
\end{equation}
where the residual $\eta_{n}$ is made of isolated terms such that $\Vert \eta_{n} \Vert = o(n)$ a.s. As an immediate consequence, we have the relation between the limiting values
\begin{equation}
\label{A21_Rel0}
\Lambda_{p}^{0} = A_{p}\, \theta^{*} + \alpha e
\end{equation}
where the almost sure limiting matrix of $\sigma^{-2} A_{n}/n$ is given by
\begin{equation}
\label{A21_P}
A_{p} = \begin{pmatrix}
\lambda_1 & \lambda_2 & \lambda_3 & \hdots & \hdots & \lambda_{p}\\
\lambda_0 & \lambda_1 & \lambda_2 & \hdots & \hdots & \lambda_{p-1}\\
\vdots & \vdots & \vdots & & & \vdots\\
 \vdots & \vdots & \vdots & & & \vdots\\
\lambda_{p-2} & \lambda_{p-3} & \lambda_{p-4} & \hdots & \hdots & \lambda_1\\
\end{pmatrix}.
\end{equation}
The reader may find more details about the way to establish these almost sure convergences \textit{e.g.} in the proof of Lemma \ref{A1_Lem_LimSn}. Likewise, one proves that
\begin{equation}
\label{A21_Rel2}
\Lambda_{p}^{2} = A_{p}^{\, \prime}\, \theta^{*} - \alpha\, \theta_{\!p} \rho J_{\! p}\, e.
\end{equation}
Finally, the very definition of the estimator $\tn$ directly implies another relation, involving the matrix $\Delta_{p}$ given by \eqref{P1_Lambda},
\begin{equation}
\label{A21_Rel1}
\Lambda_{p}^{1} = \Delta_{p}\, \theta^{*}.
\end{equation}
Relations \eqref{A21_Rel0}, \eqref{A21_Rel2} and \eqref{A21_Rel1} will be useful thereafter. Let us now consider the expression of $\rn$ given by \eqref{P2_Est}. On the one hand, in light of foregoing,
\begin{eqnarray}
\label{A21_LimNum}
\lim_{n\rightarrow \infty} \frac{1}{n} \sum_{k=1}^{n} \ek\, \eek & = & \lim_{n\rightarrow \infty} \frac{1}{n} \sum_{k=1}^{n} \Big( X_{k} - \tn^{~ \prime}\, \Phi_{k-1}^p \Big) \Big( X_{k-1} - \tn^{~ \prime}\, \Phi_{k-2}^p \Big), \nonumber\\
 & = & \sigma^2 \left( \lambda_1 - \left( {\Lambda_{p}^{0}}^{\, \prime} + {\Lambda_{p}^{2}}^{\, \prime} \right) \theta^{*}  + {\theta^{*}}^{\, \prime}\! A_{p}\, \theta^{*} \right), \nonumber \\
 & = & \sigma^2 \left( \lambda_1 - {\Lambda_{p}^{2}}^{\, \prime} \theta^{*} - \alpha \theta_1^{*} \right) \textnormal{\cvgps}
\end{eqnarray}
On the other hand, similarly,
\begin{eqnarray}
\label{A21_LimDen}
\lim_{n\rightarrow \infty} \frac{1}{n} \sum_{k=1}^{n} \eek^{~ 2} & = & \lim_{n\rightarrow \infty} \frac{1}{n} \sum_{k=1}^{n} \Big( X_{k-1} - \tn^{~ \prime}\, \Phi_{k-2}^p \Big)^2, \nonumber \\
 & = & \sigma^2 \left( \lambda_0 - 2 {\Lambda_{p}^{1}}^{\, \prime} \theta^{*}  + {\theta^{*}}^{\, \prime}\! \Delta_{p}\, \theta^{*} \right), \nonumber\\
 & = & \sigma^2 \left( \lambda_0 - {\Lambda_{p}^{1}}^{\, \prime} \theta^{*} \right) \textnormal{\cvgps}
\end{eqnarray}
\textit{Via} the set of relations \eqref{A12_SysLinLim}, we find that $\lambda_0 = \beta^{\, \prime} \Lambda_{p}^{1} - \theta_{\!p} \rho \lambda_{p+1} + 1$ for $d = 0$, and $\lambda_{p+1} = \beta^{\, \prime} J_{\! p}\, \Lambda_{p}^{1} - \theta_{\!p} \rho \lambda_{0}$ for $d = p+1$, in particular. Hence, with $\theta^{*} = \alpha (I_{\! p} - \theta_{\!p} \rho J_{\! p}) \beta$,
\begin{eqnarray}
\label{A21_Simpl1}
\lambda_1 - {\Lambda_{p}^{2}}^{\, \prime} \theta^{*} - \alpha \theta_1^{*} & = & \lambda_1 - {\Lambda_{p}^{2}}^{\, \prime} \theta^{*} - \alpha \theta_1^{*} ( \lambda_0 - \beta^{\, \prime} \Lambda_{p}^{1} + \theta_{\!p} \rho \lambda_{p+1} ), \nonumber\\
 & = & \lambda_1 - {\Lambda_{p}^{2}}^{\, \prime} \theta^{*} - \alpha \theta_1^{*} ( \lambda_0 - \beta^{\, \prime} \Lambda_{p}^{1} + \theta_{\!p} \rho ( \beta^{\, \prime} J_{\! p} \Lambda_{p}^{1} - \theta_{\!p} \rho \lambda_{0} ) ), \nonumber\\
 & = & \lambda_1 - {\Lambda_{p}^{2}}^{\, \prime} \theta^{*} - \theta_1^{*} ( \lambda_0 - {\Lambda_{p}^{1}}^{\, \prime} \theta^{*} ), \nonumber\\
 & = & \lambda_1 - {\Lambda_{p}^{2}}^{\, \prime} \theta^{*} - (\theta_1 + \rho) ( \lambda_{0} - {\Lambda_{p}^{1}}^{\, \prime} \theta^{*} ) + \theta_{\!p} \rho \theta_{\!p}^{*} ( \lambda_{0} - {\Lambda_{p}^{1}}^{\, \prime} \theta^{*} )
\end{eqnarray}
since one has to note that $\theta_1^{*} = \theta_1 + \rho - \theta_{\!p} \rho \theta_{\!p}^{*}$. \textit{Via} \eqref{A21_Rel1}, $\lambda_1 = {\Lambda_{p}^{0}}^{\, \prime} \theta^{*}$. Thus,
\begin{eqnarray}
\label{A21_Simpl2}
\lambda_1 - {\Lambda_{p}^{2}}^{\, \prime} \theta^{*} & = & {\theta^{*}}^{\, \prime} ( \Lambda_{p}^{0} - \Lambda_{p}^{2}\, ), \nonumber\\
 & = & {\theta^{*}}^{\, \prime}\! A_{p}^{\, \prime}\, \theta^{*} - {\theta^{*}}^{\, \prime}\! A_{p}\, \theta^{*} + \alpha (\theta_1 + \rho), \nonumber\\
 & = & \alpha (\theta_1 + \rho) ( \lambda_0 - \beta^{\, \prime} \Lambda_{p}^{1} + \theta_{\!p} \rho \lambda_{p+1} ), \nonumber\\
 & = & (\theta_1 + \rho) ( \lambda_{0} - {\Lambda_{p}^{1}}^{\, \prime} \theta^{*} ).
\end{eqnarray}
To conclude, \eqref{A21_Simpl1} together with \eqref{A21_Simpl2} lead to
\begin{equation*}
\lambda_1 - {\Lambda_{p}^{2}}^{\, \prime} \theta^{*} - \alpha \theta_1^{*} = \theta_{\!p} \rho \theta_{\!p}^{*} ( \lambda_{0} - {\Lambda_{p}^{1}}^{\, \prime} \theta^{*} )
\end{equation*}
which, \textit{via} \eqref{A21_LimNum} and \eqref{A21_LimDen}, achieves the proof of Theorem \ref{P2_Thm_CvgRho},
\begin{equation*}
\lim_{n\rightarrow \infty} \rn = \theta_{\!p} \rho \theta_{\!p}^{*} \textnormal{\cvgps}
\end{equation*}
\hfill
$\mathbin{\vbox{\hrule\hbox{\vrule height1ex \kern.5em\vrule height1ex}\hrule}}$

%%%%%%%%%%%%%%%%%%%%%%%%%%%%%%%%%%%%%%%%%%%%%%%%%%%%%%%%%%%%%%%%%%%%%%%%%%%%%%%%

\subsection*{}
\begin{center}
{\bf C.2. Proof of Theorem \ref{P2_Thm_TlcRho}.}
\end{center}

%%%%%%%%%%%%%%%%%%%%%%%%%%%%%%%%%%%%%%%%%%%%%%%%%%%%%%%%%%%%%%%%%%%%%%%%%%%%%%%%

First of all, we have already seen from \eqref{A11_DecompEst} that, for all $n \geq p+1$,
\begin{equation}
\label{A22_DecompTheta}
S_{n-1} \left( \tn - \theta^{*} \right) = \alpha (I_{\! p} - \theta_{\!p} \rho J_{\! p}) M_{n} + \alpha \xi_{n}
\end{equation}
where Lemma \ref{A1_Lem_StabX} involves $\Vert \xi_{n} \Vert = o(\sqrt{n})$ a.s., assuming a finite moment of order 4 for $(V_{n})$. Our goal is to find a similar decomposition for $\rn - \rho^{*}$. For a better readability, let us introduce two specific notations $Y_{n}$ and $Z_{n}$ given by
\begin{equation*}
Y_{n} = X_{n} - \rho^{*} X_{n-1} \hspace{1cm} \text{and} \hspace{1cm} Z_{n} = X_{n-1} - \rho^{*} X_{n}.
\end{equation*}
We also note $Y_{n}^{p} = \begin{pmatrix} Y_{n} & Y_{n-1} & \hdots & Y_{n-p+1} \end{pmatrix}^{\prime}$ and $Z_{n}^{p} = \begin{pmatrix} Z_{n} & Z_{n-1} & \hdots & Z_{n-p+1} \end{pmatrix}^{\prime}$. Denote by $F_{n}$ the recurrent $p-$dimensional expression that appears repeatedly in the decomposition, given, for all $n \geq 1$, by
\begin{equation}
\label{A22_Fn}
F_{n} = \Phi_{n}^p\, {\theta^{*}}^{\, \prime} Z_{n}^{p} - \left( Z_{n-1}^p + Y_{n}^p \right) X_{n}.
\end{equation}
From the residual estimation \eqref{P2_EstRes}, the development of $\rn - \rho^{*}$ reduces to
\begin{equation}
\label{A22_DecompRho}
J_{n-1} \Big( \rn - \rho^{*} \Big) = W_{n} + \left( \tn - \theta^{*} \right)^{\prime} H_{n}
\end{equation}
where $H_{n}$ is a $p-$dimensional vector and, for all $n \geq p+1$,
\begin{eqnarray}
\label{A22_Jn}
J_{n} & = & \sum_{k=0}^{n} \ek^{~ 2},\\
\label{A22_Wn}
W_{n} & = & \sum_{k=1}^{n} Z_{k} X_{k} + {\theta^{*}}^{\, \prime} \sum_{k=1}^{n} F_{k} + \nu_{n},\\
\label{A22_Hn}
H_{n} & = & \sum_{k=1}^{n} \left( Z_{k}^{p}\, {\theta^{*}}^{\, \prime}\, \Phi_{k}^p + F_{k} \right) + \sum_{k=1}^{n} \Phi_{k}^p \left( \tn - \theta^{*} \right)^{\prime} Z_{k}^{p} + \mu_{n},
\end{eqnarray}
with $\Vert \mu_{n} \Vert = o(\sqrt{n})$ a.s. and $\nu_{n} = o(\sqrt{n})$ a.s. The reasoning develops in two stages. At first, we shall prove that $W_{n}$ reduces to a martingale, except for a residual term. Then, using Theorem \ref{P1_Thm_TlcTheta} and the central limit theorem for vector martingales, we will be in the position to prove the joint asymptotic normality of our estimates.

\medskip

\noindent Let $C$ be the square submatrix of order $p+1$ obtained by removing from $B$ given by \eqref{P1_B} its first row and first column,
\begin{equation}
\label{A22_C}
C =
\begin{pmatrix}
1-\beta_2 & -\beta_3 & \hdots & \hdots & -\beta_p & \theta_{\!p} \rho & 0\\
-\beta_1-\beta_3 & \hspace{0.15cm} 1-\beta_4 \hspace{0.15cm} & \hdots & \hdots & \theta_{\!p} \rho & 0 & 0\\
\vdots & \vdots & & & \vdots & \vdots & \vdots\\
\vdots & \vdots & & & \vdots & \vdots & \vdots\\
-\beta_{p-1}+\theta_{\!p} \rho \hspace{0.15cm} & -\beta_{p-2} & \hdots & \hdots & -\beta_1 & 1 & 0\\
-\beta_p & -\beta_{p-1} & \hdots & \hdots & -\beta_2 & -\beta_1 & \hspace{0.15cm} 1 \hspace{0.15cm}\\
\end{pmatrix}.
\end{equation}

\medskip

\noindent By Corollary \ref{P1_Cor_InvC}, we have already seen that the matrix $C$ is invertible under the stability conditions. Denote by $N_{n}$ be the $(p+1)-$dimensional martingale
\begin{equation}
\label{A22_Nn}
N_{n} = \sum_{k=1}^{n} \Phi_{k-1}^{p+1} V_{k}
\end{equation}
where $\Phi_{n}^{p+1}$ stands for the extension of $\Phi_{n}^{p}$ to the next dimension. A straightforward calculation based on \eqref{A11_NewAR} shows that the following linear system is satisfied,
\begin{equation*}
C \sum_{k=1}^{n} \Phi_{k-1}^{p+1} X_{k} = T \sum_{k=1}^{n} X_{k}^2 + N_{n}
\end{equation*}
in which $T$ is defined as
\begin{equation}
\label{A22_T}
T = \begin{pmatrix} \beta_1 & \hspace{0.1cm} & \beta_2 & \hspace{0.1cm} & \hdots & \hspace{0.1cm} & \beta_{p} & \hspace{0.1cm} & -\theta_{\!p} \rho \end{pmatrix}^{\prime}.
\end{equation}
As a result of the invertibility of $C$, we get the substantial equality, for all $n \geq p+1$,
\begin{equation}
\label{A22_SommeXn}
\sum_{k=1}^{n} \Phi_{k-1}^{p+1} X_{k} = C^{-1} T \sum_{k=1}^{n} X_{k}^2 + C^{-1} N_{n}.
\end{equation}
A large manipulation of $W_{n}$ given in \eqref{A22_Wn} still based on the fundamental autoregressive form \eqref{A11_NewAR} shows, after further calculations, that there exists an isolated term $\nu_{n}$ such that $\nu_{n} = o(\sqrt{n})$ a.s., and, for all $n \geq p+1$,
\begin{eqnarray*}
W_{n} & = & \sum_{k=1}^{n} Z_{k} X_{k} - {\theta^{*}}^{\, \prime} \sum_{k=1}^{n} Z_{k-1}^p X_{k} - \alpha\, {\theta^{*}}^{\, \prime} \sum_{k=1}^{n} \left( \Phi_{k}^p - \theta_{\!p} \rho\, J_{\! p}\, \Phi_{k-2}^p \right) V_{k}\\
 & & \hspace{1cm} + \hspace{0.15cm} \alpha\, \rho^{*} {\theta^{*}}^{\, \prime} (I_{\! p} - \theta_{\!p} \rho J_{\! p}) \sum_{k=1}^{n} \Phi_{k-1}^p V_{k} + \nu_{n},
\end{eqnarray*}
leading, together with \eqref{A22_SommeXn}, to
\begin{equation}
\label{A22_SimplWn}
W_{n} = \left( G^{\, \prime} C^{-1}\, T - \rho^{*} - \alpha\, \theta_1^{*} \right) \sum_{k=1}^{n} X_{k}^2 + G^{\, \prime} C^{-1} N_{n} + L_{n} + \nu_{n}
\end{equation}
where, for all $n \geq p+1$, 
\begin{equation}
\label{A22_Ln}
L_{n} = \alpha\, {\theta^{*}}^{\, \prime} \left( \rho^{*} (I_{\! p} - \theta_{\!p} \rho J_{\! p}) M_{n} - \sum_{k=1}^{n} \left( \Phi_{k}^p - \theta_{\!p} \rho\, J_{\! p}\, \Phi_{k-2}^p \right) V_{k} \right) + \alpha\, \theta_1^{*} \sum_{k=1}^{n} X_{k} V_{k},
\end{equation}
and where the $(p+1)-$dimensional vector $G$ is given by
\begin{equation}
\label{A22_G}
G = \rho^{*} \vartheta^{*} + \alpha\, \theta_1^{*}\, T - \delta^{*}
\end{equation}
with $\vartheta^{*} = \begin{pmatrix} \theta_1^{*} & \theta_2^{*} & \hdots & \theta_{\!p}^{*} & 0 \end{pmatrix}^{\prime}$ and $\delta^{*} = \begin{pmatrix} -1 & \theta_1^{*} & \hdots & \theta_{\!p-1}^{*} & \theta_{\!p}^{*} \end{pmatrix}^{\prime}$.
In terms of almost sure limits, by using the same methodology as \textit{e.g.} in the proof of Lemma \ref{A1_Lem_LimSn}, \eqref{A22_SommeXn} directly implies
\begin{equation}
\label{A22_LimGCT}
\lambda_0\, C^{-1}\, T = \Lambda_{p+1}^{1}
\end{equation}
where $\Lambda_{p+1}^{1} = \begin{pmatrix} \lambda_1 & \lambda_2 & \hdots & \lambda_{p+1} \end{pmatrix}^{\prime}$ is the extension of $\Lambda_{p}^{1}$ in \eqref{A21_L012} to the next dimension. Hence, following the same lines as in the proof of Theorem \ref{P2_Thm_CvgRho},
\begin{eqnarray*}
\lambda_0 \left( G^{\, \prime} C^{-1}\, T - \rho^{*} - \alpha\, \theta_1^{*} \right) & = & G^{\, \prime} \Lambda_{p+1}^{1} - \lambda_0 \left( \rho^{*} + \alpha\, \theta_1^{*} \right), \\
 & = & \rho^{*} ( {\Lambda_{p}^{1}}^{\, \prime} \theta^{*} - \lambda_0 ) + \alpha\, \theta_1^{*} ( T^{\, \prime} \Lambda_{p+1}^{1} - \lambda_0) + ( \lambda_1 - {\Lambda_{p}^{2}}^{\, \prime} \theta^{*} ), \\
 & = & \theta_1^{*} ( \alpha {\Lambda_{p}^{1}}^{\, \prime} (I_{\! p} - \theta_{\!p} \rho J_{\! p} ) \beta - \alpha ( 1 - \theta_{\!p} \rho) ( 1 + \theta_{\!p} \rho) \lambda_0 ) \\
 & & \hspace{1cm} + \rho^{*} ( {\Lambda_{p}^{1}}^{\, \prime} \theta^{*} - \lambda_0 ) + ( \lambda_1 - {\Lambda_{p}^{2}}^{\, \prime} \theta^{*} ), \\
 & = & \theta_1^{*} ( {\Lambda_{p}^{1}}^{\, \prime} \theta^{*} - \lambda_0 ) + \rho^{*} ( {\Lambda_{p}^{1}}^{\, \prime} \theta^{*} - \lambda_0 ) + ( \lambda_1 - {\Lambda_{p}^{2}}^{\, \prime} \theta^{*}), \\
 & = & -\alpha (\rho^{*} + \theta_1^{*}) + \alpha (\rho^{*} + \theta_1^{*}) = 0.
\end{eqnarray*}
One can see from Lemma \ref{P1_Lem_InvL} that $\lambda_0 > 0$. The latter development ensures that the pathological term of \eqref{A22_SimplWn} vanishes, as it should. Finally, $W_{n}$ reduces to
\begin{equation}
\label{A22_SimplWn2}
W_{n} = G^{\, \prime} C^{-1} N_{n} + L_{n} + \nu_{n}, 
\end{equation}
and one shall observe that $G^{\, \prime} C^{-1} N_{n} + L_{n}$ is a locally square-integrable real martingale \cite{Duflo97}, \cite{HallHeyde80}. One is now able to combine \eqref{A22_DecompTheta} and \eqref{A22_DecompRho}, \textit{via} \eqref{A22_SimplWn2}, to establish the decomposition, for all $n \geq p+1$, 
\begin{equation}
\label{A22_DecompRho2}
J_{n-1} \Big( \rn - \rho^{*} \Big) = G^{\, \prime} C^{-1} N_{n} + L_{n} + \alpha M_{n}^{\prime} (I_{\! p} - \theta_{\!p} \rho J_{\! p}) (S_{n-1})^{\! -1} H_{n} + r_{n}
\end{equation}
where the remainder term $r_{n} = \alpha\, \xi_{n}^{\prime} (S_{n-1})^{\! -1} H_{n} + \nu_{n}$ is such that $r_{n} = o(\sqrt{n})$ a.s. Taking tediously advantage of the $(p+2) \times (p+2)$ linear system of equations \eqref{A12_SysLinLim}, one shall observe that $G^{\, \prime} C^{-1} = \alpha \begin{pmatrix} U_{p}^{\, \prime} & u_{p+1} \end{pmatrix}$ with
\begin{equation*}
U_{p} =
\begin{pmatrix}
1+\beta_2 & \hspace{0.25cm} \beta_3 - \beta_1 & \hspace{0.25cm} \hdots & \hspace{0.25cm} \beta_{p} - \beta_{p-2} & \hspace{0.25cm} -\beta_{p-1} - \theta_{\!p} \rho
\end{pmatrix}^{\prime},
\end{equation*}
and $u_{p+1} = -\alpha^{-1} \theta_{\!p}^{*} - \theta_{\!p} \rho\, \theta_{1}^{*}$. The combination of \eqref{A22_Ln} and \eqref{A22_SimplWn2} results in
\begin{equation}
\label{A22_SimplWn3}
W_{n} = \alpha \left( U_{p} + (I_{\! p} - \theta_{\!p} \rho J_{\! p}) ( \rho^{*} \theta^{*} - \tau^{*} ) \right)^{\prime} M_{n} - \theta_{\!p}^{*} \sum_{k=1}^{n} X_{k-p-1} V_{k} + \nu_{n}
\end{equation}
where $\tau^{*} = \begin{pmatrix} \theta_{2}^{*} & \theta_{3}^{*} & \hdots & \theta_{\!p}^{*} & 0 \end{pmatrix}^{\prime}$. Consequently, it follows from \eqref{A22_DecompTheta} together with \eqref{A22_DecompRho2} and \eqref{A22_SimplWn3} that
\begin{equation}
\label{A22_SystMatn}
\sqrt{n}
\begin{pmatrix}
\tn - \theta^{*}\\
\rn - \rho^{*}
\end{pmatrix} =
\frac{1}{\sqrt{n}} P_{n} N_{n} + R_{n}
\end{equation}
where the square matrix $P_{n}$ of order $p+1$ is given by
\begin{equation}
\label{A22_Pn}
P_{n} = \begin{pmatrix}
P_{n}^{(1,1)} & 0\\
P_{n}^{(2,1)} & P_{n}^{(2,2)}
\end{pmatrix}
\end{equation}
with
\begin{eqnarray*}
P_{n}^{(1,1)} & = & n (S_{n-1})^{\! -1} \alpha (I_{\! p} - \theta_{\!p} \rho J_{\! p}),\\
P_{n}^{(2,1)} & = & n (J_{n-1})^{\! -1} \left( \alpha \left( U_{p} + (I_{\! p} - \theta_{\!p} \rho J_{\! p}) ( \rho^{*} \theta^{*} - \tau^{*} ) \right)^{\prime} + \alpha H_{n}^{\, \prime} (S_{n-1})^{\! -1} (I_{\! p} - \theta_{\!p} \rho J_{\! p}) \right),\\
P_{n}^{(2,2)} & = & -n (J_{n-1})^{\! -1} \theta_{\!p}^{*},
\end{eqnarray*}
and where the $(p+1)-$dimensional remainder term
\begin{equation}
\label{A22_Rn}
R_{n} = \sqrt{n} \begin{pmatrix}
\alpha (S_{n-1})^{\! -1} \xi_{n}\\
(J_{n-1})^{\! -1} r_{n}
\end{pmatrix}
\end{equation}
is such that $\Vert R_{n} \Vert = o(1)$ a.s. Via some simplifications on $H_{n}$, \eqref{A21_Rel0}, \eqref{A21_Rel2} and \eqref{A21_Rel1}, we obtain that
\begin{equation}
\label{A22_ConvHn}
\lim_{n \rightarrow \infty} \frac{H_{n}}{n} = -\alpha (I_{\! p} - \theta_{\!p} \rho J_{\! p}) e \textnormal{\cvgps}
\end{equation}
Furthermore, it is not hard to see, \textit{via} Lemma \ref{A1_Lem_LimSn}, \eqref{A21_LimDen}, \eqref{A22_ConvHn} and some simplifications on $P_{n}^{(2,1)}$, that
\begin{equation}
\label{A22_CvgPn}
\lim_{n \rightarrow \infty} P_{n} = \sigma^{-2} P \textnormal{\cvgps}
\end{equation}
where $P$ is the limiting matrix precisely given by \eqref{P2_P}. The locally square-integrable real vector martingale $(N_{n})$ introduced in \eqref{A22_Nn} and adapted to $\cF_{n}$ has a predictable quadratic variation $\langle N \rangle_{n}$ such that
\begin{equation}
\label{A22_LimProc}
\lim_{n\rightarrow \infty} \frac{\langle N \rangle_{n}}{n} = \sigma^4 \Delta_{p+1} \textnormal{\cvgps}
\end{equation}
where $\Delta_{p+1}$ is given by \eqref{P2_LambdaP}. This convergence can be achieved following \textit{e.g.} the same lines as in the proof of Lemma \ref{A1_Lem_LimSn}. On top of that, we also immediately deduce from \eqref{A12_SommePhi4} that $(N_{n})$ satisfies the Lindeberg's condition. We conclude from the central limit theorem for martingales, given \textit{e.g.} in Corollary 2.1.10 of \cite{Duflo97}, that
\begin{equation}
\label{A22_TLC}
\frac{1}{\sqrt{n}} N_{n} \liml \cN\left( 0,\sigma^{4} \Delta_{p+1} \right).
\end{equation}
Whence, from \eqref{A22_SystMatn}, \eqref{A22_Rn}, \eqref{A22_CvgPn}, \eqref{A22_TLC} and Slutsky's lemma,
\begin{equation}
\label{A22_TLCComb}
\sqrt{n}
\begin{pmatrix}
\tn - \theta^{*}\\
\rn - \rho^{*}
\end{pmatrix} \liml \cN\left( 0, P \Delta_{p+1} P^{\, \prime} \right).
\end{equation}
This concludes the proof of Theorem \ref{P2_Thm_TlcRho} where, for readability purposes, we omitted most of the calculations which the attentive reader might easily deduce.
\hfill
$\mathbin{\vbox{\hrule\hbox{\vrule height1ex \kern.5em\vrule height1ex}\hrule}}$

%%%%%%%%%%%%%%%%%%%%%%%%%%%%%%%%%%%%%%%%%%%%%%%%%%%%%%%%%%%%%%%%%%%%%%%%%%%%%%%%

\subsection*{}
\begin{center}
{\bf C.3. Proof of Theorem \ref{P2_Thm_RatRho}.}
\end{center}

%%%%%%%%%%%%%%%%%%%%%%%%%%%%%%%%%%%%%%%%%%%%%%%%%%%%%%%%%%%%%%%%%%%%%%%%%%%%%%%%

In the proof of Theorem \ref{P2_Thm_TlcRho}, we have established a particular relation that we shall develop from now on, to achieve the proof of Theorem \ref{P2_Thm_RatRho}. Indeed, from \eqref{A22_SystMatn}, for all $n \geq p+1$,
\begin{equation}
\label{A23_DecompRho}
\rn - \rho^{*} = n^{-1} \pi_{n}^{\prime} N_{n} + (J_{n-1})^{\! -1}\, r_{n}
\end{equation}
where $N_{n}$ and $J_{n-1}$ are given by \eqref{A22_Nn} and \eqref{A22_Jn}, respectively, where $r_{n}$ is such that $r_{n} = o(\sqrt{n})$ a.s. and where $\pi_{n}$ of order $p+1$ is given from \eqref{A22_Pn} by
\begin{equation}
\label{A23_Pin}
\pi_{n} = \begin{pmatrix} P_{n}^{(2,1)} & P_{n}^{(2,2)} \end{pmatrix}^{\prime}.
\end{equation}
Denote by $\pi$ the almost sure limit of $\pi_{n}$, accordingly given by
\begin{equation}
\label{A23_Pi}
\pi = \sigma^{-2} \begin{pmatrix} P_{L}^{\prime} & \varphi \end{pmatrix}^{\prime}
\end{equation}
where $P_{L}$ and $\varphi$ are defined in \eqref{P2_P}. Hence, \eqref{A23_DecompRho} can be rewritten as
\begin{equation}
\label{A23_DecompRho2}
\rn - \rho^{*} = n^{-1} \pi^{\prime} N_{n} + n^{-1} \left( \pi_{n} - \pi \right)^{\prime} N_{n} + (J_{n-1})^{\! -1}\, r_{n}.
\end{equation}
One can note that $(\pi^{\prime} N_{n})$ is a locally square-integrable real martingale with predictable quadratic variation given, for all $n \geq 1$, by
\begin{equation}
\label{A23_Proc}
\langle \pi^{\prime} N \rangle_{n} = \sigma^2 \pi^{\prime} (T_{n-1} - T)\, \pi
\end{equation}
where the square matrix $T_{n}$ of order $p+1$ is the extension of $S_{n}$ given by \eqref{P1_Sn} to the next dimension defined, for all $n \geq 1$, as
\begin{equation}
\label{A23_Omegan}
T_{n} = \sum_{k=1}^{n} \Phi^{p+1}_{k} {\Phi^{p+1}_{k}}^{\: \prime} + T,
\end{equation}
and $T$ is a symmetric positive definite matrix. In addition, $(\pi^{\prime} N_{n})$ satisfies a nonexplosion condition summarized by
\begin{equation*}
\lim_{n\rightarrow \infty} \frac{\pi^{\prime}\, \Phi^{p+1}_{n}\, {\Phi^{p+1}_{n}}^{\: \prime} \pi}{\pi^{\prime}\, T_{n}\, \pi} = 0 \textnormal{\cvgps}
\end{equation*}
by application of Lemma \ref{A1_Lem_StabX} with $a=4$. By virtue of the quadratic strong law for martingales given \textit{e.g.} by Theorem 3 of \cite{Bercu04} or \cite{BercuCenacFayolle09},
\begin{equation}
\label{A23_LFQ}
\lim_{n\rightarrow \infty} \frac{1}{\log n} \sum_{k=1}^{n} \left( \frac{\pi^{\prime}\, N_{k}}{\pi^{\prime}\, T_{k-1}\, \pi} \right)^2 = \frac{1}{\pi^{\prime}\, \Delta_{p+1}\, \pi} \textnormal{\cvgps}
\end{equation}
where $\Delta_{p+1}$ given by \eqref{P2_LambdaP} is the almost sure limit of $\sigma^{-2} T_{n}/n$. We refer the reader to Lemma \ref{A1_Lem_LimSn} to have more details on the latter remark. Note that $\pi^{\prime}\, \Delta_{p+1}\, \pi > 0$ since $\Delta_{p+1}$ is a positive definite matrix, as a result of Lemma \ref{P1_Lem_InvL}. The same goes for $\pi^{\prime}\, T_{n}\, \pi$, for all $n \geq 1$, assuming a suitable choice of $T$. Besides, the almost sure convergence of $\pi_{n}$ to $\pi$, the finite moment of order 4 for $(V_{n})$ together with \eqref{A23_LFQ} ensure that
\begin{eqnarray}
\label{A23_Reste}
\sum_{k=1}^{n} \left( \frac{ \left( \pi_{k} - \pi \right)^{\prime} N_{k}}{k} + \frac{r_{k}}{J_{k-1}} \right)^2 & = & O\left( \sum_{k=1}^{n} \frac{ \left( \left( \pi_{k} - \pi \right)^{\prime} N_{k} \right)^2}{k^2} + \sum_{k=1}^{n} \frac{r_{k}^{\, 2}}{J_{k-1}^2} \right), \nonumber \\
 & = & O(1) + o\left( \sum_{k=1}^{n} \frac{\left( \pi^{\prime} N_{k} \right)^2}{k^2} \right), \nonumber \\
 & = & o(\log n) \textnormal{\cvgps}
\end{eqnarray}
since $r_{n}$ is made of isolated terms of order 2 and $J_{n} = O(n)$ a.s. It follows that
\begin{eqnarray*}
\lim_{n\rightarrow \infty} \frac{1}{\log n} \sum_{k=1}^{n} \Big( \rk - \rho^{*} \Big)^2 & = & \lim_{n\rightarrow \infty} \frac{1}{\log n} \sum_{k=1}^{n} \left( \frac{\pi^{\prime} N_{k}}{k} \right)^2, \\
 & = & \lim_{n\rightarrow \infty} \frac{1}{\log n} \sum_{k=1}^{n} \left( \frac{\pi^{\prime} N_{k}}{\pi^{\prime}\, T_{k-1}\, \pi} \right)^2 \left( \frac{\pi^{\prime}\, T_{k-1}\, \pi}{k} \right)^2, \\
 & = & \frac{\sigma^4 (\pi^{\prime}\, \Delta_{p+1}\, \pi)^2}{\pi^{\prime}\, \Delta_{p+1}\, \pi} = \sigma^4 \pi^{\prime}\, \Delta_{p+1}\, \pi \textnormal{\cvgps}
\end{eqnarray*}
\textit{via} \eqref{A23_LFQ} and \eqref{A23_Reste}, since the cross-term also plays a negligible role compared to the leading one. The definition of $\pi$ in \eqref{A23_Pi} combined with the one of $\Gamma$ in \eqref{P2_GamCov} achieves the proof of the first part of Theorem \ref{P2_Thm_RatRho}.

\medskip

\noindent Furthermore, it follows from the law of iterated logarithm for martingales \cite{Stout70}, \cite{Stout74}, see also Corollary 6.4.25 of \cite{Duflo97}, that
\begin{eqnarray*}
\limsup_{n \rightarrow \infty} \left( \frac{\langle \pi^{\prime} N \rangle_{n}}{2 \log \log \langle \pi^{\prime} N \rangle_{n}} \right)^{\! 1/2} \frac{\pi^{\prime} N_{n}}{\langle \pi^{\prime} N \rangle_{n}} & = & - \liminf_{n \rightarrow \infty} \left( \frac{\langle \pi^{\prime} N \rangle_{n}}{2 \log \log \langle \pi^{\prime} N \rangle_{n}} \right)^{\! 1/2} \frac{\pi^{\prime} N_{n}}{\langle \pi^{\prime} N \rangle_{n}}, \\
 & = & 1 \textnormal{\cvgps}
\end{eqnarray*}
since we have \textit{via} \eqref{A13_H3} that
\begin{equation}
\label{A23_CondLIL}
\sum_{k=1}^{\infty} \frac{(\pi^{\prime} \Phi_{k-1}^{p})^4}{k^2} < +\infty \textnormal{\cvgps}
\end{equation}
Recall that we have the almost sure convergence
\begin{equation}
\label{A23_LIL1}
\lim_{n\rightarrow \infty} \frac{\langle \pi^{\prime} N \rangle_{n}}{n} = \sigma^4 \pi^{\prime} \Delta_{p+1} \pi \textnormal{\cvgps}
\end{equation}
Therefore, we immediately obtain that
\begin{eqnarray}
\limsup_{n \rightarrow \infty} \left( \frac{n}{2 \log \log n} \right)^{\! 1/2} \frac{\pi^{\prime} N_{n}}{\langle \pi^{\prime} N \rangle_{n}} & = & - \liminf_{n \rightarrow \infty} \left( \frac{n}{2 \log \log n} \right)^{\! 1/2} \frac{\pi^{\prime} N_{n}}{\langle \pi^{\prime} N \rangle_{n}}, \nonumber \\
\label{A23_LIL2}
 & = & \sigma^{-2} (\pi^{\prime} \Delta_{p+1} \pi)^{\! -1/2}  \textnormal{\cvgps}
\end{eqnarray}
As in the previous proof and by virtue of the same arguments, one can easily be convinced that the remainder term in the right-hand side of \eqref{A23_DecompRho2} is negligible. It follows from \eqref{A23_DecompRho2} together with \eqref{A23_LIL1} and \eqref{A23_LIL2} that,
\begin{eqnarray*}
\limsup_{n \rightarrow \infty} \left( \frac{n}{2 \log \log n} \right)^{\! 1/2} \Big( \rn - \rho^{*} \Big) & = & - \liminf_{n \rightarrow \infty} \left( \frac{n}{2 \log \log n} \right)^{\! 1/2} \Big( \rn - \rho^{*} \Big), \\
 & = & \sigma^{2} \sqrt{ \pi^{\prime} \Delta_{p+1} \pi }  \textnormal{\cvgps}
\end{eqnarray*}
which achieves the proof of Theorem \ref{P2_Thm_RatRho}.
\hfill
$\mathbin{\vbox{\hrule\hbox{\vrule height1ex \kern.5em\vrule height1ex}\hrule}}$

%%%%%%%%%%%%%%%%%%%%%%%%%%%%%%%%%%%%%%%%%%%%%%%%%%%%%%%%%%%%%%%%%%%%%%%%%%%%%%%%

\bigskip

\section*{\textcolor{blue}{Appendix D}}

\begin{center}
{\small \textcolor{blue}{COMPARISON WITH THE H-TEST OF DURBIN}}
\end{center}

\renewcommand{\thesection}{\Alph{section}} 
\renewcommand{\theequation}
{\thesection.\arabic{equation}} \setcounter{section}{4}  
\setcounter{equation}{0}
\setcounter{lem}{0}

%%%%%%%%%%%%%%%%%%%%%%%%%%%%%%%%%%%%%%%%%%%%%%%%%%%%%%%%%%%%%%%%%%%%%%%%%%%%%%%%

\medskip

\textcolor{blue}{We shall now compare our statistical procedure with the well-known \textit{h-test} of Durbin \cite{Durbin70}. Let us assume that $\cH_0$ is true, that is $\rho=0$. Then, the least squares estimate of the variance of $\tn$ is given by
\begin{equation}
\label{HT_EstVar}
\wh{\dV}_{n}(\tn) = \wh{\sigma}_{n}^{\, 2}\, S_{n-1}^{-1}
\end{equation}
where $S_{n}$ is given in \eqref{P1_Sn} and $\wh{\sigma}_{n}^{\, 2}$ is the strongly consistent least squares estimate of $\sigma^2$ under $\cH_0$, defined as
\begin{equation}
\label{HT_EstRes}
\wh{\sigma}_{n}^{\, 2} = \frac{1}{n} \sum_{k=0}^{n} \ek^{~ 2}.
\end{equation}
For this proof, we use a Toeplitz version of $S_{n}$ given by
\begin{equation*}
S_{n}^{\, p} = \begin{pmatrix}
s_{n}^{\, 0} & s_{n}^{\, 1} & s_{n}^{\, 2} & \hdots & s_{n}^{\, p-1} \\
s_{n}^{\, 1} & s_{n}^{\, 0} & s_{n}^{\, 1} & \hdots & s_{n}^{\, p-2} \\
s_{n}^{\, 2} & s_{n}^{\, 1} & s_{n}^{\, 0} & \hdots & s_{n}^{\, p-3} \\
\vdots & \vdots & \vdots & \ddots & \vdots \\
s_{n}^{\, p-1} & s_{n}^{\, p-2} & s_{n}^{\, p-3} & \hdots & s_{n}^{\, 0}
\end{pmatrix}
\end{equation*}
where, for all $0 \leq h \leq p$, $s_{n}^{\, h} = \sum_{k=0}^{n} X_{k} X_{k-h}$, and we easily note that $S_{n}^{\, p} = S_{n} + o(n)$ a.s. We assume for the sake of simplicity that $S_{n}^{\, p}$ is invertible, saving us from adding a positive definite matrix $S$ without loss of generality. We also define
\begin{equation*}
\Pi_{n}^{h} = \Big( s_{n}^{\, 1} \hspace{0.3cm} s_{n}^{\, 2} \hspace{0.3cm} \hdots \hspace{0.3cm} s_{n}^{\, h} \Big)^{\prime} \hspace{0.5cm} \text{and} \hspace{0.5cm} \wh{\vartheta}_{n}^{\, p-1} = \begin{pmatrix} \wh{\vartheta}_{1,\, n} & \wh{\vartheta}_{2,\, n} & \hdots & \wh{\vartheta}_{p-1,\, n} \end{pmatrix}^{\prime}
\end{equation*}
with $\Pi_{n} = \Pi_{n}^{p}$, $\pi_{n} = \Pi_{n}^{p-1}$ and $\wh{\vartheta}_{n} = (S_{n}^{\, p})^{-1}\, \Pi_{n}$ is the Yule-Walker estimator. First, a simple calculation from \eqref{HT_EstRes} shows that
\begin{equation}
\label{HT_Sig2}
n\, \wh{\sigma}_{n}^{\, 2} = s_{n}^{\, 0} - \Pi_{n}^{\, \prime}\, \wh{\vartheta}_{n} \vspace{0.2cm}
\end{equation}
where $\wh{\sigma}_{n}^{\, 2}$ is built from $\wh{\vartheta}_{n}$. In addition, the first diagonal element of $(S_{n}^{\, p})^{-1}$ is the inverse of the Schur complement of $S_{n}^{\, p-1}$ in $S_{n}^{\, p}$, given by
\begin{equation}
\label{HT_InvSn11}
s_{n}^{\, 0} - \pi_{n}^{\, \prime}\, (\, S_{n}^{\, p-1} \,)^{-1}\, \pi_{n}.  \vspace{0.2cm}
\end{equation}
The conjunction of \eqref{HT_Sig2} and \eqref{HT_InvSn11} leads to
\begin{equation}
\label{HT_Part1}
1 - n \wh{\dV}_{n}(\wh{\vartheta}_{1,\, n}) = \frac{\alpha_{n} - \beta_{n}}{\alpha_{n}}
\end{equation}
with
\begin{equation*}
\alpha_{n} = s_{n}^{\, 0} - \pi_{n}^{\, \prime}\, (\, S_{n}^{\, p-1} \,)^{-1}\, \pi_{n} \hspace{0.5cm} \text{and} \hspace{0.5cm} \beta_{n} = s_{n}^{\, 0} - \Pi_{n}^{\, \prime}\, (\, S_{n}^{\, p} \,)^{-1}\, \Pi_{n}. \vspace{0.2cm}
\end{equation*}
We also easily establish, \textit{via} some straightforward calculations, that
\begin{equation*}
\pi_{n} = k_{n} \left( I_{\! p-1} +  \wh{\vartheta}_{p,\, n}\, J_{\! p-1} \right) S_{n}^{\, p-1}\, \wh{\vartheta}^{\, p-1}_{n} \hspace{0.5cm} \text{with} \hspace{0.5cm} k_{n} = \left( 1 - \wh{\vartheta}_{p,\, n}^{~ 2} \right)^{\!-1}
\end{equation*}
leading, since $S_{n}^{\, p-1}$ is bissymetric and commutes with $J_{\! p-1}$, to
\begin{equation*}
\alpha_{n} = s_{n}^{\, 0} - k_{n}\, \pi_{n}^{\, \prime}\, \wh{\vartheta}^{\, p-1}_{n} - k_{n}\, \wh{\vartheta}_{p,\, n}\, \pi_{n}^{\, \prime}\, J_{\! p-1}\, \wh{\vartheta}^{\, p-1}_{n} \hspace{0.5cm} \text{and} \hspace{0.5cm} \pi_{n}^{\, \prime}\, J_{\! p-1}\, \wh{\vartheta}^{\, p-1}_{n} = s_{n}^{\, p} - \wh{\vartheta}_{p,\, n}\, s_{n}^{\, 0}. \vspace{0.2cm}
\end{equation*}
Hence, from the previous results,
\begin{eqnarray}
\label{HT_LienAlphaBeta}
k_{n}^{-1}\, \alpha_{n} & = & k_{n}^{-1} \left( s_{n}^{\, 0} - k_{n}\, \pi_{n}^{\, \prime}\, \wh{\vartheta}^{\, p-1}_{n} - k_{n}\, \wh{\vartheta}_{p,\, n}\, s_{n}^{\, p} + k_{n}\, \wh{\vartheta}_{p,\, n}^{~ 2}\, s_{n}^{\, 0} \right),\nonumber \\
 & = & s_{n}^{\, 0} - \pi_{n}^{\, \prime}\, \wh{\vartheta}^{\, p-1}_{n} - \wh{\vartheta}_{p,\, n}\, s_{n}^{\, p},\nonumber \\
 & = & s_{n}^{\, 0} - \Pi_{n}^{\, \prime}\, \wh{\vartheta}_{n} = \beta_{n}. \vspace{0.2cm}
\end{eqnarray}
We now easily conclude from \eqref{HT_Part1} and \eqref{HT_LienAlphaBeta} that
\begin{equation*}
1 - n \wh{\dV}_{n}(\wh{\vartheta}_{1,\, n}) = \wh{\vartheta}_{p,\, n}^{~ 2}.
\end{equation*}
Considering now that $(\, S_{n}^{\, p} \,)^{-1} = S_{n}^{-1} + o(n^{-1})$ a.s. and making use of $\tn$ given by \eqref{P1_Est}, it is straightforward to obtain that $\tn = \wh{\vartheta}_{n} + o(1)$ a.s. and that
\begin{equation*}
1 - n \wh{\dV}_{n}(\wh{\theta}_{1,\, n}) = \wh{\theta}_{p,\, n}^{~ 2} + o(1) \cvgps
\end{equation*}
which ends the proof.} \hfill
$\mathbin{\vbox{\hrule\hbox{\vrule height1ex \kern.5em\vrule height1ex}\hrule}}$

%%%%%%%%%%%%%%%%%%%%%%%%%%%%%%%%%%%%%%%%%%%%%%%%%%%%%%%%%%%%%%%%%%%%%%%%%%%%%%%%

\bigskip

\noindent{\bf Acknowledgments.} \textit{The author thanks Bernard Bercu for all his advices and suggestions during the preparation of this work. \textcolor{blue}{The author also thanks the Associate Editor and the two anonymous Reviewers for their suggestions and constructive comments which helped to improve the paper substantially.}}

\nocite{*}

\bibliographystyle{acm}
\bibliography{ARPDW2011}

\medskip

\end{document}